\documentclass[11pt,a4paper]{amsart}
\usepackage{amsfonts}
\usepackage{}
\usepackage{amsfonts}
\usepackage{amsmath,xypic}
\usepackage{mathrsfs}
\usepackage{amssymb}

\usepackage{CJK,CJKnumb}
\usepackage[CJKbookmarks,colorlinks,
            linkcolor=black,
            anchorcolor=black,
            citecolor=blue]{hyperref}
\usepackage{color}              
\usepackage{indentfirst}        
\usepackage{latexsym,bm}        
\usepackage{amsmath,amssymb}    
\usepackage{pstricks}
\usepackage{pst-node}
\usepackage{pst-tree}
\usepackage{pst-plot}
\usepackage{pst-text}
\usepackage{graphicx}
\usepackage{cases}
\usepackage{pifont}
\usepackage{txfonts}
\usepackage[all,knot,poly]{xy}

\allowdisplaybreaks[4]  

\CJKtilde   

\newcommand{\msu}{\mathscr{U}}
\newcommand{\al}{\alpha}
\newcommand{\be}{\beta}
\newcommand{\de}{\delta}
\newcommand{\De}{\Delta}

\newcommand{\ga}{\gamma}

\newcommand{\bC}{\mathbb{C}}
\newcommand{\bN}{\mathbb{N}}
\newcommand{\bQ}{\mathbb{Q}}
\newcommand{\bK}{\mathbb{K}}
\newcommand{\bZ}{\mathbb{Z}}

\newcommand{\ot}{\otimes}

\newcommand{\op}{\oplus}
\newcommand{\si}{\sigma}

\newcommand{\ve}{\varepsilon}

\newcommand{\ma}{\mathscr{A}}
\newcommand{\mh}{\mathscr{H}}
\newcommand{\mb}[1]{\mbox{#1}}

\newcommand{\m}[1]{\mathcal{#1}}
\newcommand{\mi}{\mbox{id}}
\newcommand{\bs}[1]{{\scriptsize\mbox{#1}}}

\newcommand{\til}{\tilde{N}}
\newcommand{\stt}[1]{{\scriptstyle #1}}

\newcommand{\ol}[1]{\overline{#1}}
\newcommand{\ull}[1]{\underline{#1}}

\newcommand{\lan}{\langle}
\newcommand{\ran}{\rangle}
\newcommand{\lb}{\left(}
\newcommand{\rb}{\right)}
\newcommand{\lc}{\lceil}
\newcommand{\rc}{\rceil}
\newcommand{\lf}{\lfloor}
\newcommand{\rf}{\rfloor}
\newcommand{\rw}{\rightarrow}

\begin{document}

\newtheorem{theorem}{Theorem}[section]

\newtheorem{lem}[theorem]{Lemma}

\newtheorem{cor}[theorem]{Corollary}
\newtheorem{prop}[theorem]{Proposition}

\theoremstyle{remark}
\newtheorem{rem}[theorem]{Remark}

\newtheorem{defn}[theorem]{Definition}

\newtheorem{exam}[theorem]{Example}

\title[The Green rings of $\bar{\ma}_{-1}(2)$ and its two relatives twisted]
{The Green rings of the $2$-rank Taft algebra\\ and its two relatives twisted}
\author[Li]{Yunnan Li}
\address{Department of Mathematics, Shanghai Key Laboratory of Pure Mathematics and Mathematical Practice, East China Normal University,
Minhang Campus, Dong Chuan Road 500, Shanghai 200241, PR China}
\email{yunnan814@163.com}

\author[Hu]{Naihong Hu$^\star$}
\address{Department of Mathematics, Shanghai Key Laboratory of Pure Mathematics and Mathematical Practice, East China Normal University,
Minhang Campus, Dong Chuan Road 500, Shanghai 200241, PR China}
\email{nhhu@math.ecnu.edu.cn}
\thanks{$^\star$N. H.,
the corresponding author, supported in part by the NNSFC (Grant No.
 11271131)}

\date\today

\begin{abstract}
In the paper, the representation rings (or the Green rings) for a family of Hopf algebras of tame type, the $2$-rank Taft algebra (at $q=-1$) and its two relatives twisted by $2$-cocycles are
explicitly described via a representation theoretic analysis. It turns out that the Green rings can serve to detect effectively the
twist-equivalent Hopf algebras here.
\end{abstract}
\maketitle

\section{Introduction}
\subsection{}
It is well-known that Drinfeld twist is a key method to yielding new Hopf algebras in quantum groups theory (see \cite{D, GM}, etc.). Dually, $2$-cocycle twist or Doi-Majid twist including Drinfeld double as a kind of such twist (see \cite{dt, M}) is extensively employed in various current researches.
For instance, Andruskiewitsch et al (\cite{AFGV}) considered the twists of Nichols algebras associated to racks and cocycles. Guillot-Kassel-Masuoka (\cite{GKM}) got some examples by twisting comodule algebras by $2$-cocycles.
In generic case, Pei-Hu-Rosso (\cite{PHR}), Hu-Pei (\cite{hp}) found explicit $2$-cocycle deformation formulae between multi-(resp. two-)parameter quantum groups and one-parameter quantum groups $U_{q,q^{-1}}(\mathfrak g)$ and an equivalence between the weight module categories $\mathcal O$ as braided tensor ones. Likewise, in root of unity case, when a $2$-cocycle twist exists under some conditions on the parameters, Benkart et al (\cite{bpw}) used a result of Majid-Oeckl (\cite{mo}) to give a category equivalence between Yetter-Drinfeld modules for a finite-dimensional pointed Hopf algebra $H$ and
those for its cocycle twist $H^\sigma$, and further to derive an equivalence of the categories of modules for $\mathfrak u_{r,s}(\mathfrak {sl}_n)$ and $\mathfrak u_{q,q^{-1}}(\mathfrak {sl}_n)$ as Drinfeld doubles. In contrast, for particular choices of the parameters, there is no such cocycle twist, and in that
situation the representation theories of $\mathfrak u_{r,s}(\mathfrak {sl}_n)$ and $\mathfrak u_{q,q^{-1}}(\mathfrak {sl}_n)$ can be quite different (see Example 5.6 of \cite{bpw}).
Recently, Bazlov-Berenstein considered cocycle twists and extensions of braided doubles in a broader setting including twisting the rational Cherednik algebra of the symmetric group into the Spin Cherednik algebra (see \cite{BB}).

A natural question is to ask how to detect two twist-equivalent Hopf algebras in nature? The article seeks to address this question through investigating the representation rings for a family of Hopf algebras, the $2$-rank Taft Hopf algebra (at $q=-1$) of which we introduced in \cite{hu1} before, and its two relatives twisted by $2$-cocycles.

\subsection{}
Given a Hopf algebra $H$, in the investigation of its monoidal module category, the decomposition problem of tensor products of indecomposables is of most importance and has received enormous attentions. One main approach is to explore the ring structure of the corresponding representation ring or say the Green ring of $H$. Originally, the concept of the Green ring $r(H)$ stems from the modular representations of finite groups (see \cite{Gre} etc.).
Since then, there are plenty of works on the Green rings. For finite groups, one can refer to the papers of Green \cite{Gre1}, Benson \cite{BC,BP}, etc. We also mention that Witherspoon computed the Green ring of the Drinfeld double of a finite group in \cite{With}. Chen-Oystaeyen-Zhang studied the Green rings of Taft algebras in \cite{COZ}. On the other hand, Cibils defined a quiver quantum group $\bK Z_n(q)/I_d$ and considered the decomposition of tensor products in \cite{Cib1}. This quiver quantum group is isomorphic to the generalized Taft algebra $H_{n,d}$ defined in \cite{HCZ}. Recently, in order to investigate the Green ring of $H_{n,d}$, Li-Zhang (\cite{LZ}) reformulated the decomposition formulas given by Cibils in \cite{Cib1}. They determined all nilpotent elements in $r(H_{n,d})$. Wakui (\cite{Wa}) also described the representation ring structures for all eight dimensional nonsemisimple Hopf algebras of finite type except for the only one of tame type.
Erdmann et al (\cite{EGST}) determined a large part of the structure of the Green
ring of $D(\Lambda_{n,d})$ modulo projectives where $D(\Lambda_{n,d})$ stands for the Drinfeld doubles of a family of duals of generalized Taft Hopf algebras $\Lambda_{n,d}$ using the different method.

\subsection{}
In \cite{Cib}, Cibils showed that one half of small quantum groups are
all wild when the rank $n\geq2$ and $\mbox{ord}(q)\geq5$. By contrast,
the (generalized) Taft algebras are of finite type. On the other hand,
Feldvoss-Witherspoon (\cite{FW}) proved a conjecture due to Cibils using support variety stating that
all small quantum groups of rank at least two are wild. For the rank one case,
as we know, $\mathfrak{u}_q(\mathfrak{s}\mathfrak{l}_2)$ is tame, the study on indecomposables and their tensor product
decomposition is already perfect (cf. \cite{CP, KS, Sut, Xiao}, etc).

In \cite{hu,hu1}, the second author defined the quantum divided power algebras and the quantized enveloping algebras of abelian Lie algebras. Denote by $\ma_q(n)$ the latter. When $q$ is a root of unity, $\ma_q(n)$ admits a finite-dimensional Hopf quotient. We call it the {\it $n$-rank Taft algebra} (or small abelian quantum group), denoted $\bar{\ma}_q(n)$. When $n=1$, it comes back to the famous Taft algebra (see \cite{Ta}).
In what follows, we will concern the objects of tame type among the $n$-rank Taft (Hopf) algebras. According to the discussion of Ringel in \cite{rin}, the only one being tame is the $2$-rank Taft algebra
$\bar{\ma}_q(2)$ at $q=-1$.
Briefly set $\bar{\ma}:=\bar{\ma}_{-1}(2)$.

\subsection{}
The paper is organized as follows. In Section 2, we begin by giving the definition of the $2$-rank Taft Hopf algebra $\bar{\ma}$ and a complete set of orthogonal primitive idempotents with the Gabriel quiver, and then describe
its indecomposables. We decompose the various tensor products of indecomposables in Section 3. This leads to the description of the Green ring of $\bar{\ma}$, its Jacobson radical and the projective class algebra in Section 4.
Section 5 continues to give two $2$-cocycle twisted Hopf algebras $\bar{\ma}^{\sigma_i}$: $D(H_4)$ and $\mh$ $(=H_4\ot H_4)$, where $H_4$ is the Sweedler Hopf algebra of dimension $4$, and explicitly determine their Green rings via a similar representation-theoretic analysis, together with the Jacobson radicals, the projective class algebras, etc.
It is interesting to notice that even the Hopf algebras $\mh$, $\bar{\ma}$ and $D(H_4)$ are twist-equivalent to each other and are of dimension $16$, they own the different number of blocks with $1$, $2$ and $3$, respectively, whose diverse information on the Green rings are listed in the end of the paper.

As another evidence, we mention the work of Caenepeel-Dascalescu-Raianu (see \cite{CDR}) classifying all pointed Hopf algebras of dimension $16$.
After finishing the paper, we happen to find (and by comparison with those) that the Hopf algebras in question are the exact $3$ iso-classes of pointed Hopf algebras with the Klein group algebra
as the coradicals among the five iso-classes in their classification list, however, the rest do not exist any twist-equivalence with others.
In general, it is not clear how the $2$-cocycle twist of the multiplication
affects the representation ring.  Hopefully, it will stimulate a further research.

Throughout, we work over an algebraically closed field $\bK$ of characteristic $0$. Unless otherwise stated, all (Hopf) algebras and modules defined over $\bK$ are finite-dimensional. Given an algebra $A$, let $A$-mod denote the category of finite-dimensional left $A$-modules.

\section{The $2$-rank Taft algebra $\bar{\ma}$ and its indecomposable modules}
\subsection{}
From \cite[(1.1)(a),(c),(9)]{rin}, we know that the algebras $\bK[x,y,z]/(x^2,y^2,z^2,xy,yz,xz)$, $\bK[x,y]/(x^a,y^b),~a,b\geq2$ are of infinite representation type. Explicitly, it is tame if $a=b=2$ and wild if $a,b\geq2$ but not both equal to 2. It means that if we want to figure out the representation ring of $\bar{\ma}_q(n)$, only the case for $q=-1$ and $n=2$ is reachable when $n\geq2$.

Now we describe the $2$-rank Taft Hopf algebra $\bar{\ma}$ in detail. $\bar{\ma}$ has generators $g, \;h,\; x,\; y$, subject to the following relations,
\begin{equation}\label{rel}
\begin{array}{l}
gh=hg, \quad g^2=h^2=1,\\
gx=-xg,\quad gy=-yg,\quad hx=-xh,\quad hy=-yh,\\
x^2=y^2=0,\quad xy=-yx.
\end{array}
\end{equation}
The comultiplication $\De$ is defined by
\[\De(g)=g\ot g,\quad \De(h)=h\ot h,\quad \De(x)=x\ot 1+g\ot x,\quad \De(y)=y\ot 1+h\ot y.\]

$\bar{\ma}$ has four orthogonal primitive idempotents
\[e_1=\tfrac{1}{4}(1+g+h+gh), \ e_2=\tfrac{1}{4}(1+g-h-gh), \
e_3=\tfrac{1}{4}(1-g+h-gh), \ e_4=\tfrac{1}{4}(1-g-h+gh),\]
and two central primitive idempotents
\[f_1=e_1+e_4=\tfrac{1}{2}(1+gh),\quad f_2=e_2+e_3=\tfrac{1}{2}(1-gh).\]
\begin{lem}
(1) There exist two signs $s_1, s_2\in\{+,-\}$ such that $ge_i=s_1e_i$, $he_i=s_2e_i$ with $(s_1,s_2)=(+,+),(+,-),(-,+),(--)$, for $i=1, 2, 3, 4$, successively.

(2) \ $xe_1=e_4x,\quad  xe_4=e_1x,\quad  ye_1=e_4y,\quad  ye_4=e_1y,\quad  xe_2=e_3x,\quad  xe_3=e_2x$,

\hskip.68cm $ye_2=e_3y,\quad  ye_3=e_2y$.
\end{lem}
\begin{proof}
It is straightforward to check.
\end{proof}

\subsection{}
Let $S(s_1,s_2)$ denote the one dimensional simple module $\bK$ of $\bar{\ma}$, defined by $g\cdot1=s_11,h\cdot1=s_21,x\cdot1=y\cdot1=0$.
Let $P(s_1,s_2)$ be the projective cover of $S(s_1,s_2)$, which coincides with the principle indecomposable module
$\bar{\ma}e_i$ for some $i\in\{1,2,3,4\}$. As $P(s_1,s_2)$ are non-isomorphic to each other with respect to the signs $s_1,s_2$,
$\bar{\ma}$ is a basic algebra over $\bK$. On the other hand, the radical $J(\bar{\ma})=(x,y)$, thus from the above lemma,
we know that the Gabriel quiver $Q_{\bar{\ma}}$ of $\bar{\ma}$ looks like:
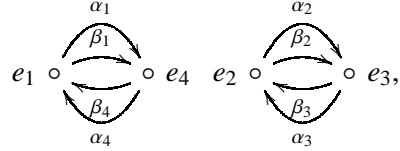
\[e_1\xymatrix@=2em{
\circ\ar@{->}@/^/[r]^-{\be_1}\ar@{->}@/^{1.7em}/[r]^-{\al_1}&\circ\ar@{->}@/^/[l]^-{\be_4}\ar@{->}@/^{1.7em}/[l]^-{\al_4}}e_4~~~
e_2\xymatrix@=2em{\circ\ar@{->}@/^/[r]^-{\be_2}\ar@{->}@/^{1.7em}/[r]^-{\al_2}&\circ\ar@{->}@/^/[l]^-{\be_3}\ar@{->}@/^{1.7em}/[l]^-{\al_3}}e_3,\]
where for $i=1,2,3,4$, the arrows $\al_i,\be_i$ correspond to $xe_i,\, ye_i$, respectively. The admissible ideal $I$ has the following relations:
\[\begin{array}{l}
\al_1\al_4=\al_4\al_1=0,\quad \al_2\al_3=\al_3\al_2=0,\quad
\be_1\be_4=\be_4\be_1=0,\quad \be_2\be_3=\be_3\be_2=0,\\
\be_4\al_1+\al_4\be_1=\be_1\al_4+\al_1\be_4=
\be_3\al_2+\al_3\be_2=\be_2\al_3+\al_2\be_3=0.
\end{array}
\]
Hence, $\bar{\ma}$ decomposes into two block $\bar{\ma}f_1,\bar{\ma}f_2$.
Their representation categories can transfer to each other via the functor $\cdot\ot S(+,-)$.

From the quiver $Q_{\bar{\ma}}$, we know that $\bar{\ma}$ is a special biserial algebra.
All indecomposable modules of such kind of algebras can be completely described.
For the whole theory of special biserial algebras, we refer to \cite[II]{Erd}.
Now we only focus on our target $\bar{\ma}$.

Note that the simple modules $S(s_1,s_2),s_1,s_2\in\{+,-\}$ exhaust all simple modules of $\bar{\ma}$,
thus the projective modules $P(s_1,s_2),s_1,s_2\in\{+,-\}$
are all indecomposable projective modules of $\bar{\ma}$. Moreover, $P(s_1,s_2)\cong P(+,+)\ot S(s_1,s_2)$.
Here we highlight $S(+,+),S(-,-),P(+,+),P(-,-)$ and write them as $1,S,P,P_-$, respectively for short.
The module structure of $P$ can be presented by the following diagram:
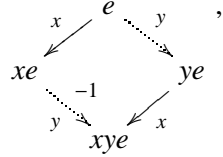
\[\xymatrix@=1em{
&{e}\ar@{->}[ld]_-{x} \ar@{.>}[rd]^-{y}&\\
{xe}\ar@{.>}[rd]_-{y}^-{-1}&&{ye}\ar@{->}[ld]^-{x}\\
&{xye}&},\]
where the arrow $\xy0;/r.15pc/:{\ar@{->}(-5,0)*{};(5,0)*{}};(0,2)*{\stt{x}};\endxy$ (resp. $\xy0;/r.15pc/:{\ar@{.>}(-5,0)*{};(5,0)*{}};(0,2)*{\stt{y}};\endxy$)
 represents the action of $x$ (resp. $y$) on the module.
 We abbreviate them to $\xy0;/r.15pc/:{\ar@{->}(-5,0)*{};(5,0)*{}};\endxy$ and $\xy0;/r.15pc/:{\ar@{.>}(-5,0)*{};(5,0)*{}};\endxy$.
 Given $M\in\bar{\ma}\mb{-mod}$, for any $a\in\bK,u,v\in M$, we use $\xy0;/r.15pc/:{\ar@{->}(-5,0)*{};(5,0)*{}};(0,2)*{\stt{a}};(-7,0)*{u};(7,0)*{v};\endxy$ (resp. $\xy0;/r.15pc/:{\ar@{.>}(-5,0)*{};(5,0)*{}};(0,2)*{\stt{a}};(-7,0)*{u};(7,0)*{v};\endxy$) to represent $x\cdot u=av$ (resp. $y\cdot u=av$).
 Moreover, we omit the decoration of the arrow if the weight $a=1$.

Note that any finite dimensional indecomposable module of $\bar{\ma}$ has the Loewy length at most 3. $P(s_1,s_2),~s_1,s_2=\pm$
exhaust all indecomposable modules whose Loewy length are 3. From the theory of special biserial algebras, those with the Loewy length $2$
can be divided into string modules and band modules as follows.

There exist five groups of indecomposable modules: $\{M(r)\}_{r\in\bZ^+}$, $\{W(r)\}_{r\in\bZ^+}$,
$\{N(r)\}_{r\in\bZ^+}$, $\{N'(r)\}_{r\in\bZ^+}$, $\{C(r,\eta)\}_{r\in\bZ^+,\eta\in \bK^\times}$,
where the first four groups are all string modules, and the last one are band modules.
$M(r)=(\op_{i=1}^r \bK u_i)\op(\op_{i=1}^{r+1} \bK v_i)$ is defined by the following diagram:
\[\xymatrix@=1em{
&u_1\ar@{->}[ld] \ar@{.>}[rd]& &u_2\ar@{->}[ld] &\cdots&u_r\ar@{->}[ld] \ar@{.>}[rd]&\\
v_1&&v_2&\cdots&v_r&&v_{r+1}},\]
which is equivalent to the string module $M((\al_1\be_1^{-1})^r)$.

$W(r)=(\op_{i=1}^{r+1} \bK u_i)\op(\op_{i=1}^r \bK v_i)$ is defined by the following diagram:
\[\xymatrix@=1em{
u_1\ar@{->}[rd]& &u_2\ar@{.>}[ld] \ar@{->}[rd] &\cdots&u_r\ar@{->}[rd]&&u_{r+1}\ar@{.>}[ld]\\
&v_1&&v_2&\cdots&v_r&},\]
which is equivalent to the string module $M((\al_1^{-1}\be_1)^r)$.

$N(r)=(\op_{i=1}^r \bK u_i)\op(\op_{i=1}^r \bK v_i)$ is defined by the following diagram:
\[\xymatrix@=1em{
u_1\ar@{->}[rd]&&u_2\ar@{.>}[ld] \ar@{->}[rd] &\cdots&u_{r-1}\ar@{->}[rd]&&u_r\ar@{.>}[ld] \ar@{->}[rd]&\\
&v_1&&v_2&\cdots&v_{r-1}&&v_r},\]
which is equivalent to the string module $M((\al_1^{-1}\be_1)^{r-1}\al_1^{-1})$.

$N'(r)=(\op_{i=1}^r \bK u_i)\op(\op_{i=1}^r \bK v_i)$ is defined by the following diagram:
\[\xymatrix@=1em{
u_1\ar@{.>}[rd]&&u_2\ar@{->}[ld] \ar@{.>}[rd] &\cdots&u_{r-1}\ar@{.>}[rd]&&u_r\ar@{->}[ld] \ar@{.>}[rd]&\\
&v_1&&v_2&\cdots&v_{r-1}&&v_r},\]
which is equivalent to the string module $M((\be_1^{-1}\al_1)^{r-1}\be_1^{-1})$.

$C(r,\eta)=(\op_{i=1}^r \bK u_i)\op(\op_{i=1}^r \bK v_i)$ is defined by the following diagram:
\[\xymatrix@=1em{
u_1\ar@{.>}@/_/[rd]_-{\eta}\ar@{->}@/^/[rd]&&u_2\ar@{.>}[ld] \ar@{.>}@/_/[rd]_-{\eta}\ar@{->}@/^/[rd] &\cdots&u_{r-1}\ar@{.>}@/_/[rd]_-{\eta}\ar@{->}@/^/[rd]&&u_r\ar@{.>}[ld] \ar@{.>}@/_/[rd]_-{\eta}\ar@{->}@/^/[rd]&\\
&v_1&&v_2&\cdots&v_{r-1}&&v_r},\]
which is equivalent to the band module $M(\be_1\al_1^{-1},r,\eta)$. Here $y\cdot u_i=\eta v_i+v_{i-1},~i=1,\dots,r$ and we set $v_0=0$, for convenience.

Meanwhile, for any $M\in\mbox{ind}(\bar{\ma})$ in the five groups above, we fix the trivial actions of $g,h$ on $u_1$ and extend it to the whole modules naturally.
To change the diagonal actions of $g,h$ on $M$, one only needs to consider the tensor product of $M$ with some $S(s_1,s_2)$.
In particular, we also use the notation $M(r)_-,\,W(r)_-,\,C(r,\eta)_-,\,N(r)_-,\,N'(r)_-$, when tensoring with $S(-,-)$. For instance, $C(r,\eta)_-$ corresponds to the band module $M(\be_4\al_4^{-1},r,\eta)$.
Referring to \cite[II]{Erd}, we get

\begin{theorem}\label{ind}
The $5$ groups of modules $\{M(r)\}_{r\in\mathbb{Z}^+}$,
$\{W(r)\}_{r\in\mathbb{Z}^+}$, $\{C(r,\eta)\}_{r\in\mathbb{Z}^+,\eta\in \bK^\times}$, $\{N(r)\}_{r\in\mathbb{Z}^+}$, $\{N'(r)\}_{r\in\mathbb{Z}^+}$,
together with $\{1,P\}$ provide a complete list of isomorphism classes of finite dimensional indecomposable modules of $\bar{\ma}$
when tensoring with all the simple modules $S(s_1,s_2)$.
\end{theorem}

\section{Monoidal category $\bar{\ma}$-$\mbox{mod}$}
\subsection{}
In this section, we decompose all the tensor products of indecomposable modules listed in Theorem \ref{ind}.
In \cite{Wa}, the author considered the quasitriangularity and the Green rings of all $8$-dimensional nonsemisimple Hopf algebras over $\bK$,
except the unique tame case, $A_{C_2}$. $A_{C_2}$ also appears as a special case $U_{(2,1,\sqrt{-1})}$ in Gelaki's paper \cite{Gel}.
Note that
 $A_{C_2\times C_2},A_{C_2}$ in Table 1 of \cite{Wa} nicely serve as the Hopf quotients $\bar{\ma}/(y),\bar{\ma}/(g-h)$,
  both of which are self-dual and quasi-triangular. Meanwhile, one should observe that the ideal $(g-h)$ annihilates the following modules,
\[\m{M}=\big\{\,1,\,S,\,P_\pm,\,M(r)_\pm,\,W(r)_\pm,\,N(r)_\pm,\,N'(r)_\pm,\,
C(r,\eta)_\pm\,\mid\,r\in\mathbb{Z}^+,\eta\in \bK^\times\,\big\}.\]
Now we denote $\bar{\ma}':=\bar{\ma}/(g-h)$, then
 $\m{M}$ provides all the indecomposable objects in the representation category $\bar{\ma}'\mb{-mod}$, which is a full subcategory of $\bar{\ma}\mb{-mod}$ via the pull-back.

\subsection{}
In order to give all the tensor product decompositions for $\bar{\ma}$-mod, we first give the following result.
\begin{lem}\label{sim}
For any $M\in \{1,\,P,\,M(r),\,W(r),\,N(r),\,N'(r),\,
\mid r\in\mathbb{Z}^+\}$ and signs $s_1, s_2=\pm$, we have
\[M\ot S(s_1,s_2)\cong S(s_1,s_2)\ot M.\]
By comparison,
\[S(s_1,s_2)\ot C(r,\eta)\cong C(r,s_1s_2\eta)\ot S(s_1,s_2),\quad r\in\mathbb{Z}^+,\eta\in \bK^\times.\]
\end{lem}
\begin{proof}
First we abbreviate any $1\ot u\in S(s_1,s_2)\ot M,\,u\ot1\in M\ot S(s_1,s_2)$ as $u$. Then for $S(s_1,s_2)\ot P$, we have the diagram
\[\xymatrix@=1em{
&{e}\ar@{->}[ld]\ar@{.>}[rd]&\\
{s_1xe}\ar@{.>}[rd]_-{-1}&&{s_2ye}\ar@{->}[ld]\\
&{s_1s_2xye}&},\]
where the signs $s_1,s_2$ serve as $\pm1$. Now since $e$ has the sign $(s_1,s_2)$, we get $S(s_1,s_2)\ot P\cong P\ot S(s_1,s_2)$. For $S(s_1,s_2)\ot M(r)$, we have
\[\xymatrix@=1em{
&u_1\ar@{->}[ld] \ar@{.>}[rd]& &s_1s_2u_2\ar@{->}[ld] \ar@{.>}[rd]&&u_3\ar@{->}[ld] \ar@{.>}[rd]&\cdots&&\\
s_1v_1&&s_2v_2&&s_1v_3&&s_2v_4&\cdots}\]
such that $u_1$ has the sign $(s_1,s_2)$, thus the isomorphism holds. The cases for $W(r),N(r),N(r)'$ are quite similar. For $S(s_1,s_2)\ot C(r,\eta)$, we have the diagram
\[\xymatrix@=1em{
u_1\ar@{.>}@/_/[rd]_-{s_1s_2\eta}\ar@{->}@/^/[rd]
&&s_1s_2u_2\ar@{.>}[ld] \ar@{.>}@/_/[rd]_-{s_1s_2\eta}\ar@{->}@/^/[rd] &&u_3\ar@{.>}[ld] \ar@{.>}@/_/[rd]_-{s_1s_2\eta}\ar@{->}@/^/[rd]
&&s_1s_2u_4\ar@{.>}[ld]&\cdots\\
&s_1v_1&&s_2v_2&&s_1v_3&\cdots&}\]
such that $u_1$ has the sign $(s_1,s_2)$, thus $S(s_1,s_2)\ot C(r,\eta)\cong C(r,s_1s_2\eta)\ot S(s_1,s_2)$.
\end{proof}

Now we only need to focus on the decompositions of modules in \[\big\{\,P,\,M(r),\,W(r),\,N(r),\,N'(r)\mid r\in\mathbb{Z}^+\,\big\}.\]
Note that their tensor products commute as $\bar{\ma}'$ is quasi-triangular,
hence we only need to provide a one-sided version of decompositions of tensor products of them.

\begin{theorem}\label{dec}
For any $r,s\in\mathbb{Z}^+,\eta,\ga\in \bK^\times$, we have

(1) $P\ot P\cong P^{\op2}\op P_-^{\op2}$.

(2) $M(r)\ot P\cong P^{\op r}\op P_-^{\op r+1},\quad W(r)\ot P\cong P^{\op r+1}\op P_-^{\op r}$.

(3) $C(r,\eta)\ot P\cong N(r)\ot P\cong N'(r)\ot P\cong P^{\op r}\op P_-^{\op r}$.

(4) $M(r)\ot M(s)\cong P^{\op rs}\op M(r+s)_-$.

(5) $W(r)\ot W(s)\cong P^{\op rs}\op W(r+s)$.

(6) $M(r)\ot W(s)\cong\begin{cases}
P^{\op r(s+1)}\op W(s-r)_-,& r<s,\\
P^{\op r(r+1)}\op S,& r=s,\\
P^{\op (r+1)s}\op M(r-s),& r>s.
\end{cases}$

(7) $M(r)\ot C(s,\eta)\cong P^{\op rs}\op C(s,\eta)_-$.

(8) $W(r)\ot C(s,\eta)\cong P^{\op rs}\op C(s,\eta)$.

(9) $M(r)\ot N(s)\cong  P^{\op rs}\op N(s)_-,\quad M(r)\ot N'(s)\cong P^{\op rs}\op N'(s)_-$.

(10) $W(r)\ot N(s)\cong P^{\op rs}\op N(s),\quad W(r)\ot N'(s)\cong P^{\op rs}\op N'(s)$.

(11) $C(r,\eta)\ot C(s,\ga)\cong
 \begin{cases}
   P^{\op rs},&\eta\neq\ga,\\
   P^{\op(rs-\bs{min}\{r,s\})}\op C(\mb{min}\{r,s\},\eta)\op C(\mb{min}\{r,s\},\eta)_-,&\eta=\ga.
 \end{cases}$

(12) $N(r)\ot N(s)\cong P^{\op(rs-\bs{min}(r,s))}\op N(\mbox{min}(r,s))\op N(\mbox{min}(r,s))_-$.

(13) $N'(r)\ot N'(s)\cong P^{\op(rs-\bs{min}(r,s))}\op N'(\mbox{min}(r,s))\op N'(\mbox{min}(r,s))_-$.

(14) $N(r)\ot N'(s)\cong C(r,\eta)\ot N(s)\cong C(r,\eta)\ot N'(s)\cong P^{\op rs}$.

\end{theorem}
\begin{proof}
In order to simplify the proof, we only provide the diagrams for all the module structures.
The elements in the second position of a tensor product will be added a superscript $'$, thus we can distinguish the notations.

(1) \[\begin{array}{l}
\raisebox{2.4em}
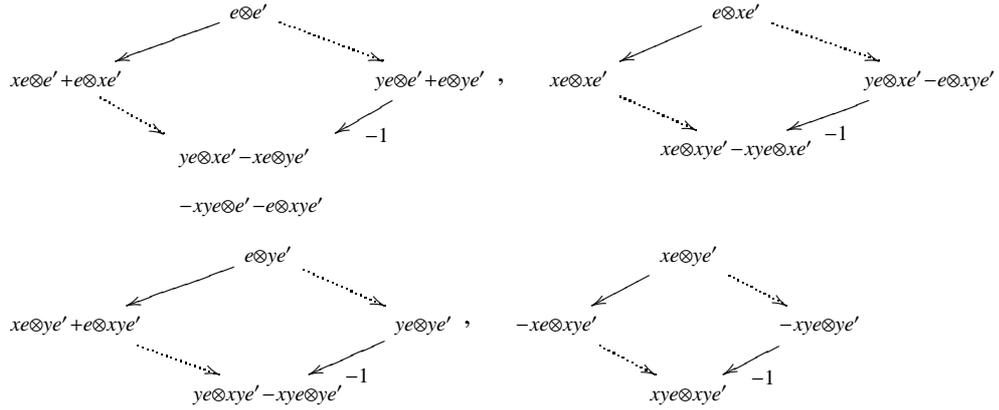
{\xymatrix@=1em{
&{\stt{e\ot e'}}\ar@{->}[ld] \ar@{.>}[rd]&\\
{\stt{xe\ot e'+e\ot xe'}}\ar@{.>}[rd]&&{\stt{ye\ot e'+e\ot ye'}}\ar@{->}[ld]^-{-1}\\
&{\begin{split}
&\stt{ye\ot xe'-xe\ot ye'}\\
&\stt{-xye\ot e'-e\ot xye'}
\end{split}}&}},\quad
\raisebox{2.4em}{\xymatrix@=1em{
&{\stt{e\ot xe'}}\ar@{->}[ld] \ar@{.>}[rd]&\\
{\stt{xe\ot xe'}}\ar@{.>}[rd]&&{\stt{ye\ot xe'-e\ot xye'}}\ar@{->}[ld]^-{-1}\\
&{\stt{xe\ot xye'-xye\ot xe'}}&}}\\
\raisebox{2.4em}{\xymatrix@=1em{
&{\stt{e\ot ye'}}\ar@{->}[ld] \ar@{.>}[rd]&\\
{\stt{xe\ot ye'+e\ot xye'}}\ar@{.>}[rd]&&{\stt{ye\ot ye'}}\ar@{->}[ld]^-{-1}\\
&{\stt{ye\ot xye'-xye\ot ye'}}&}},\quad
\raisebox{2.4em}{\xymatrix@=1em{
&{\stt{xe\ot ye'}}\ar@{->}[ld] \ar@{.>}[rd]&\\
{\stt{-xe\ot xye'}}\ar@{.>}[rd]&&{\stt{-xye\ot ye'}}\ar@{->}[ld]^-{-1}\\
&{\stt{xye\ot xye'}}&}}
\end{array}\]
Since the collection of vectors in all the vertices forms a basis of $P\ot P$, we get the desired decomposition $P^{\op 2}\op P_-^{\op 2}$.
All the decompositions below can be read similarly.

(2) In the case of $M(r)\ot P$, we have
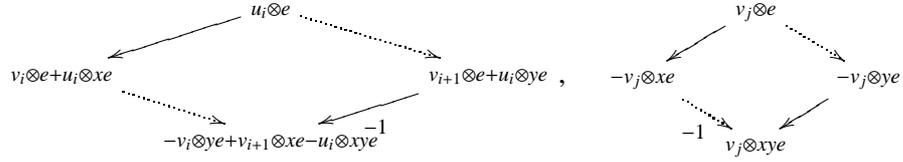
\[\raisebox{2.6em}{\xymatrix@=1em{
&{\stt{u_i\ot e}}\ar@{->}[ld] \ar@{.>}[rd]&\\
{\stt{v_i\ot e+u_i\ot xe}}\ar@{.>}[rd]&&{\stt{v_{i+1}\ot e+u_i\ot ye}}\ar@{->}[ld]^-{-1}\\
&{\stt{-v_i\ot ye+v_{i+1}\ot xe-u_i\ot xye}}&}},\quad\raisebox{2.6em}{\xymatrix@=1em{
&{\stt{v_j\ot e}}\ar@{->}[ld] \ar@{.>}[rd]&\\
{\stt{-v_j\ot xe}}\ar@{.>}[rd]_-{-1}&&{\stt{-v_j\ot ye}}\ar@{->}[ld]\\
&{\stt{v_j\ot xye}}&}}\]
for $i=1,\dots,r,j=1,\dots,r+1$, while in the case of $W(r)\ot P$,
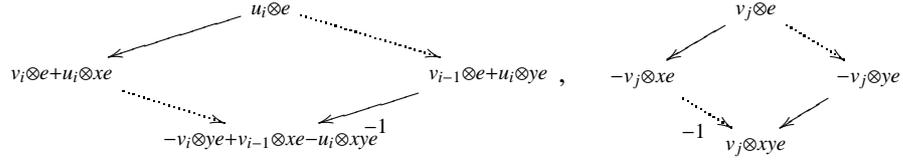
\[\raisebox{2.6em}{\xymatrix@=1em{
&{\stt{u_i\ot e}}\ar@{->}[ld] \ar@{.>}[rd]&\\
{\stt{v_i\ot e+u_i\ot xe}}\ar@{.>}[rd]&&{\stt{v_{i-1}\ot e+u_i\ot ye}}\ar@{->}[ld]^-{-1}\\
&{\stt{-v_i\ot ye+v_{i-1}\ot xe-u_i\ot xye}}&}},\quad\raisebox{2.6em}{\xymatrix@=1em{
&{\stt{v_j\ot e}}\ar@{->}[ld] \ar@{.>}[rd]&\\
{\stt{-v_j\ot xe}}\ar@{.>}[rd]_-{-1}&&{\stt{-v_j\ot ye}}\ar@{->}[ld]\\
&{\stt{v_j\ot xye}}&}}\]
for $i=1,\dots,r+1,j=1,\dots,r$, where we take $v_0=0$ for simplicity.

(3) For the case $C(r,\eta)\ot P$, we have
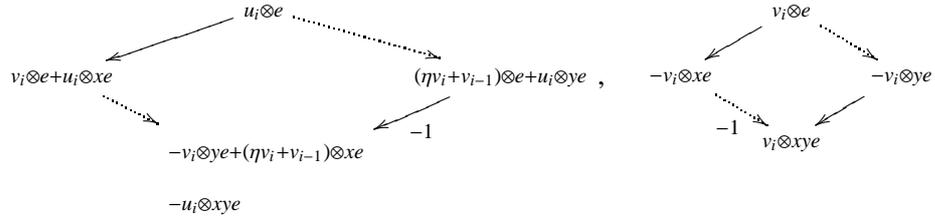
\[\raisebox{2.6em}{\xymatrix@=1em{
&{\stt{u_i\ot e}}\ar@{->}[ld] \ar@{.>}[rd]&\\
{\stt{v_i\ot e+u_i\ot xe}}\ar@{.>}[rd]&&{\stt{(\eta v_i+v_{i-1})\ot e+u_i\ot ye}}\ar@{->}[ld]^-{-1}\\
&{\begin{split}
&\stt{-v_i\ot ye+(\eta v_i+v_{i-1})\ot xe}\\
&\stt{-u_i\ot xye}
\end{split}}&}},\quad\raisebox{2.6em}{\xymatrix@=1em{
&{\stt{v_i\ot e}}\ar@{->}[ld] \ar@{.>}[rd]&\\
{\stt{-v_i\ot xe}}\ar@{.>}[rd]_-{-1}&&{\stt{-v_i\ot ye}}\ar@{->}[ld]\\
&{\stt{v_i\ot xye}}&}}\]
where $i=1,\dots,r$. For $N(r)\ot P$, it has the same diagram as that of $M(r)\ot P$ but with $i=j=1,\cdots,r$,
while $N'(r)\ot P$ is similar to that of $W(r)\ot P$.

(4) We first prove the base case $M(1)\ot M(s)$ as follows.
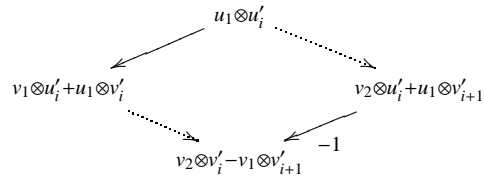
\[\xymatrix@=1em{
&{\stt{u_1\ot u'_i}}\ar@{->}[ld] \ar@{.>}[rd]&\\
{\stt{v_1\ot u'_i+u_1\ot v'_i}}\ar@{.>}[rd]&&{\stt{v_2\ot u'_i+u_1\ot v'_{i+1}}}\ar@{->}[ld]^-{-1}\\
&{\stt{v_2\ot v'_i-v_1\ot v'_{i+1}}}&}\]
for $i=1,\dots,s$. Meanwhile,
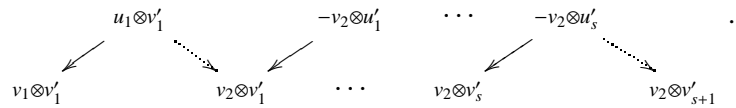
\[\xymatrix@=1em{
&\stt{u_1\ot v'_1}\ar@{->}[ld] \ar@{.>}[rd]& &\stt{-v_2\ot u'_1}\ar@{->}[ld] &\cdots&\stt{-v_2\ot u'_s}\ar@{->}[ld] \ar@{.>}[rd]&\\
\stt{v_1\ot v'_1}&&\stt{v_2\ot v'_1}&\cdots&\stt{v_2\ot v'_s}&&\stt{v_2\ot v'_{s+1}}}.\]
Now we prove the general case $M(r)\ot M(s)$ by induction on $r$. From (2)  and the base case, we see that
\[(M(1)\ot M(r)) \ot M(s)\cong(P^{\op r}\op M(r{+}1)_-)\ot M(s)\cong P^{\op rs}\op P_-^{\op r(s{+}1)}\op(M(r{+}1)_-\ot M(s)).\]
On the other hand, using (2) and the induction hypothesis, we have
\[M(1)\ot(M(r)\ot M(s))\cong M(1)\ot(P^{\op rs}\op M(r{+}s)_-)\cong (P^{\op rs}\op P_-^{\op 2rs})\op(P_-^{\op r+s}\op M(r{+}s{+}1)).\]
By the Krull-Schmidt theorem, they combine to give $M(r{+}1)\ot M(s)\cong P^{\op(r{+}1)s}\op M(r{+}s{+}1)_-$.

(5) The base case $W(1)\ot W(s)$ is as follows.
\[\xymatrix@=1em{
&{\stt{u_1\ot u'_i}}\ar@{->}[ld] \ar@{.>}[rd]&\\
{\stt{u_1\ot v'_i+v_1\ot u'_i}}\ar@{.>}[rd]_-{-1}&&{\stt{u_1\ot v'_{i-1}}}\ar@{->}[ld]\\
&{\stt{v_1\ot v'_{i-1}}}&}\]
for $i=2,\dots,s+1$, where we take $v'_{s+1}=0$ for simplicity. Meanwhile,
\[\xymatrix@C=0.2em{
\stt{u_1\ot u'_1}\ar@{->}[rd]& &\stt{u_1\ot u'_2+u_2\ot u'_1}\ar@{.>}[ld] \ar@{->}[rd] &\cdots&\stt{u_1\ot u'_{s+1}+u_2\ot u'_s}\ar@{->}[rd]&&\stt{u_2\ot u'_{s+1}}\ar@{.>}[ld]\\
&\stt{u_1\ot v'_1+v_1\ot u'_1}&&\stt{v_1\ot u'_2+u_1\ot v'_2+u_2\ot v'_1}&\cdots&\stt{v_1\ot u'_{s+1}+u_2\ot v'_s}&}.\]
Now we prove the general case $W(r)\ot W(s)$ by induction on $r$. From (2)  and the base case, we see that
\[(W(1)\ot W(r)) \ot W(s)\cong(P^{\op r}\op W(r{+}1))\ot W(s)\cong P^{\op r(s+1)}\op {P_-}^{\op rs}\op(W(r{+}1)\ot W(s)).\]
On the other hand, using (2) and the induction hypothesis, we have
\[W(1)\ot(W(r)\ot W(s))\cong W(1)\ot(P^{\op rs}\op W(r{+}s))\cong (P^{\op 2rs}\op P_-^{\op rs})\op(P^{\op r+s}\op W(r{+}s{+}1)).\]
By the Krull-Schmidt theorem, they combine to give $W(r{+}1)\ot W(s)\cong P^{\op(r+1)s}\op W(r{+}s{+}1)$.

(6) Here we only deal with the case $r\leq s$. When $r>s$, the proof is similar. First, the base case $M(1)\ot W(1)$ is as follows.
\[\raisebox{2.6em}{\xymatrix@=1em{
&{\stt{u_1\ot u'_1}}\ar@{->}[ld] \ar@{.>}[rd]&\\
{\stt{v_1\ot u'_1+u_1\ot v'_1}}\ar@{.>}[rd]&&{\stt{v_2\ot u'_1}}\ar@{->}[ld]^-{-1}\\
&{\stt{v_2\ot v'_1}}&}},\quad\raisebox{2.6em}{\xymatrix@=1em{
&{\stt{u_1\ot u'_2}}\ar@{->}[ld] \ar@{.>}[rd]&\\
{\stt{v_1\ot u'_2}}\ar@{.>}[rd]_-{-1}&&{\stt{v_2\ot u'_2+u_1\ot v'_1}}\ar@{->}[ld]\\
&{\stt{v_1\ot v'_1}}&}}.\]
Meanwhile, $v_1\ot u'_1+u_1\ot v'_1+v_2\ot u'_2$ spans the simple module $S$.

Next, we consider $M(1)\ot W(s),~s>1$ as follows. By the base case and (2),
we have
\[(M(1)\ot W(1))\ot W(s{-}1)\cong(P^{\op2}\op S)\ot W(s{-}1)\cong P^{\op 2s}\op P_-^{\op 2(s-1)}\op W(s{-}1)_-.\]
On the other hand, from the base case, (2) and (5), we get
\[M(1)\ot(W(1)\ot W(s{-}1))\cong M(1)\ot(P^{\op s-1}\op W(s))\cong(P^{\op s-1}\op P_-^{\op 2(s-1)})\op(M(1)\ot W(s)).\]
They combine to give $M(1)\ot W(s)\cong P^{\op s+1}\op W(s{-}1)_-$.

Now we decompose $M(r)\ot W(s),~r\leq s$ by induction on $r$. First, we have
\[(M(1)\ot M(r{-}1))\ot W(s)\cong(P^{\op r-1}\op M(r)_-)\ot W(s)\cong P^{\op(r-1)(s{+}1)}\op P_-^{\op(r-1)s}\op(M(r)_-\ot W(s)).\]
On the other hand, by induction, we have
\[\begin{split}
M(1)\ot(M(r{-}1)\ot W(s))&\cong M(1)\ot(P^{\op(r-1)(s{+}1)}\op W(s{-}r{+}1)_-)\\
&\cong P^{\op(r-1)(s{+}1)}\op P_-^{\op2(r-1)(s{+}1)}\op(M(1)\ot W(s{-}r{+}1)_-).
\end{split}\]
They combine to give
\[M(r)\ot W(s)\cong P^{\op(r-1)(s{+}2)}\op(M(1)\ot W(s{-}r{+}1))\cong
\begin{cases}
P^{\op r(s+1)}\op W(s{-}r)_-,& r<s,\\
P^{\op r(r+1)}\op S,& r=s.
\end{cases}\]

(7) The induction step on $r$ is quite similar to (4) via the associativity. We only give the base case $M(1)\ot C(s,\eta)$ as follows,
\[\xymatrix@=1em{
&{\stt{u_1\ot u'_i}}\ar@{->}[ld] \ar@{.>}[rd]&\\
{\stt{u_1\ot v'_i+v_1\ot u'_i}}\ar@{.>}[rd]&&{\stt{u_1\ot(\eta v'_i+v'_{i-1})+v_2\ot u'_i}}\ar@{->}[ld]^-{-1}\\
&{\stt{-v_1\ot(\eta v'_i+v'_{i-1})+v_2\ot v'_i}}&}\]
for $i=1,\dots,s$. Meanwhile,
\[\xymatrix@=1em{
\stt{v_1\ot u'_1}\ar@{.>}@/_/[rd]_-{\eta}\ar@{->}@/^/[rd]&&\stt{v_1\ot u'_2}\ar@{.>}[ld] \ar@{.>}@/_/[rd]_-{\eta}\ar@{->}@/^/[rd] &\cdots&\stt{v_1\ot u'_{s-1}}\ar@{.>}@/_/[rd]_-{\eta}\ar@{->}@/^/[rd]&&\stt{v_1\ot u'_s}\ar@{.>}[ld] \ar@{.>}@/_/[rd]_-{\eta}\ar@{->}@/^/[rd]&\\
&\stt{-v_1\ot v'_1}&&\stt{-v_1\ot v'_2}&\cdots&\stt{-v_1\ot v'_{s-1}}&&\stt{-v_1\ot v'_s}}.\]

(8) The induction step on $r$ is quite similar to (5), we only give the base case $W(1)\ot C(s,\eta)$ as follows.
\[\xymatrix@=1em{
&{\stt{u_1\ot u'_i}}\ar@{->}[ld] \ar@{.>}[rd]&\\
{\stt{u_1\ot v'_i+v_1\ot u'_i}}\ar@{.>}[rd]_-{-1}&&{\stt{u_1\ot(\eta v'_i+v'_{i-1})}}\ar@{->}[ld]\\
&{\stt{v_1\ot(\eta v'_i+v'_{i-1})}}&}\]
for $i=1,\dots,s$. Meanwhile,
\[\xymatrix@=1em{
\stt{U_1}\ar@{.>}@/_/[rd]_-{\eta}\ar@{->}@/^/[rd]&&\stt{U_2}\ar@{.>}[ld] \ar@{.>}@/_/[rd]_-{\eta}\ar@{->}@/^/[rd] &\cdots&\stt{U_{s-1}}\ar@{.>}@/_/[rd]_-{\eta}\ar@{->}@/^/[rd]&&\stt{U_s}\ar@{.>}[ld] \ar@{.>}@/_/[rd]_-{\eta}\ar@{->}@/^/[rd]&\\
&\stt{V_1}&&\stt{V_2}&\cdots&\stt{V_{s-1}}&&\stt{V_s}},\]
where $U_i=u_2\ot(\eta u'_i+u'_{i-1})+u_1\ot u'_i,V_i=u_2\ot(\eta v'_i+v'_{i-1})+v_1\ot u'_i+u_1\ot v'_i,~i=1,\dots,s$. Note that $y\cdot U_i=\eta V_i+V_{i-1}$.


(9) We only give the base case $M(1)\ot N(s)$ as follows.
\[\xymatrix@=1em{
&{\stt{u_1\ot u'_i}}\ar@{->}[ld] \ar@{.>}[rd]&\\
{\stt{u_1\ot v'_i+v_1\ot u'_i}}\ar@{.>}[rd]&&{\stt{u_1\ot v'_{i-1}+v_2\ot u'_i}}\ar@{->}[ld]^-{-1}\\
&{\stt{-v_1\ot v'_{i-1}+v_2\ot v'_i}}&}\]
for $i=1,\dots,s$, where we take $v'_0=0$ for simplicity. Meanwhile,
\[\xymatrix@=1em{
\stt{v_1\ot u'_1}\ar@{->}[rd]&&\stt{v_1\ot u'_2}\ar@{.>}[ld] \ar@{->}[rd] &\cdots&\stt{v_1\ot u'_{s-1}}\ar@{->}[rd]&&\stt{v_1\ot u'_s}\ar@{.>}[ld] \ar@{->}[rd]&\\
&\stt{v_1\ot v'_1}&&\stt{v_1\ot v'_2}&\cdots&\stt{v_1\ot v'_{s-1}}&&\stt{v_1\ot v'_s}}.\]
The case of $M(1)\ot N'(s)$ is quite similar.

(10) The base case $W(1)\ot N(s)$ is as follows.
\[\raisebox{2.6em}{\xymatrix@=1em{
&{\stt{u_1\ot u'_i}}\ar@{->}[ld] \ar@{.>}[rd]&\\
{\stt{u_1\ot v'_i+v_1\ot u'_i}}\ar@{.>}[rd]_-{-1}&&{\stt{u_1\ot v'_{i-1}}}\ar@{->}[ld]\\
&{\stt{v_1\ot v'_{i-1}}}&}},\quad\raisebox{2.6em}{\xymatrix@=1em{
&{\stt{u_2\ot u'_s}}\ar@{->}[ld] \ar@{.>}[rd]&\\
{\stt{u_2\ot v'_s}}\ar@{.>}[rd]&&{\stt{v_1\ot u'_s+u_2\ot v'_{s-1}}}\ar@{->}[ld]^-{-1}\\
&{\stt{v_1\ot v'_s}}&}}\]
for $i=2,\dots,s$. Meanwhile,
\[\xymatrix@C=-2em{
\stt{u_1\ot u'_1}\ar@{->}[rd]&&\stt{u_1\ot u'_2+u_2\ot u'_1}\ar@{.>}[ld] \ar@{->}[rd] &\cdots&\stt{u_1\ot u'_{s-1}+u_2\ot u'_{s-2}}\ar@{->}[rd]&&\stt{u_1\ot u'_s+u_2\ot u'_{s-1}}\ar@{.>}[ld] \ar@{->}[rd]&\\
&\stt{u_1\ot v'_1+v_1\ot u'_1}&&\stt{u_1\ot v'_2+v_1\ot u'_2+u_2\ot v'_1}&\cdots&\stt{u_1\ot v'_{s-1}+v_1\ot u'_{s-1}+u_2\ot v'_{s-2}}&&\stt{u_1\ot v'_s+v_1\ot u'_s+u_2\ot v'_{s-1}}}.\]
The case of $W(1)\ot N'(s)$ is quite similar.

(11) When $\eta\neq\ga$, the decomposition of $C(r,\eta)\ot C(s,\ga)$ is given by:
\[\xymatrix@=0.3em{
&{\stt{u_i\ot u'_j}}\ar@{->}[ld] \ar@{.>}[rd]&\\
{\stt{v_i\ot u'_j+u_i\ot v'_j}}\ar@{.>}[rd]_-{-1}&&{\begin{split}
&\stt{(\eta v_i+v_{i-1})\ot u'_j}\\
&\stt{+u_i\ot (\ga v'_j+v'_{j-1})}
\end{split}}\ar@{->}[ld]\\
&{\begin{split}
&\stt{-(\eta v_i+v_{i-1})\ot v'_j}\\
&\stt{+v_i\ot (\ga v'_j+v'_{j-1})}
\end{split}
}&},\]
for $i=1,\dots,r,j=1,\dots,s$.

When $\eta=\ga$, assume $r\leq s$ without loss of generality, then $C(r,\eta)\ot C(s,\ga)$ decomposes as follows:
\[\xymatrix@=1em{
\stt{U_1}\ar@{.>}@/_/[rd]_-{\eta}\ar@{->}@/^/[rd]&&\stt{U_2}\ar@{.>}[ld] \ar@{.>}@/_/[rd]_-{\eta}\ar@{->}@/^/[rd] &\cdots&\stt{U_{r-1}}\ar@{.>}@/_/[rd]_-{\eta}\ar@{->}@/^/[rd]&&\stt{U_r}\ar@{.>}[ld] \ar@{.>}@/_/[rd]_-{\eta}\ar@{->}@/^/[rd]&\\
&\stt{V_1}&&\stt{V_2}&\cdots&\stt{V_{r-1}}&&\stt{V_r}},\]
where $U_k=\sum_{i=1}^ku_i\ot u'_{k+1-i},V_k=\sum_{i=1}^k v_i\ot u'_{k+1-i}+u_i\ot v'_{k+1-i},~k=1,\dots,r$. And $y\cdot U_i=\eta V_i+V_{i-1}$. Meanwhile,
\[\xymatrix@=1em{
\stt{u_1\ot v'_s}\ar@{.>}@/_/[rd]_-{\eta}\ar@{->}@/^/[rd]&&\stt{u_2\ot v_s}\ar@{.>}[ld] \ar@{.>}@/_/[rd]_-{\eta}\ar@{->}@/^/[rd] &\cdots&\stt{u_{r-1}\ot v'_s}\ar@{.>}@/_/[rd]_-{\eta}\ar@{->}@/^/[rd]&&\stt{u_r\ot v'_s}\ar@{.>}[ld] \ar@{.>}@/_/[rd]_-{\eta}\ar@{->}@/^/[rd]&\\
&\stt{v_1\ot v'_s}&&\stt{v_2\ot v'_s}&\cdots&\stt{u_{r-1}\ot v'_s}&&\stt{u_r\ot v'_s}}.\]
Moreover, we take
\[\xymatrix@=0.3em{
&{\stt{u_i\ot u'_j}}\ar@{->}[ld] \ar@{.>}[rd]&\\
{\stt{v_i\ot u'_j+u_i\ot v'_j}}\ar@{.>}[rd]_-{-1}&&{\begin{split}
&\stt{\eta(v_i\ot u'_j+u_i\ot v'_j)}\\
&\stt{+v_{i-1}\ot u'_j+u_i\ot v'_{j-1}}
\end{split}}\ar@{->}[ld]\\
&{\stt{v_{i-1}\ot v'_j-v_i\ot v'_{j-1}}}&}\]
for $i=1,\dots,r,j=2,\dots,s$. Then
$C(r,\eta)\ot C(s,\eta)\cong P^{\op(r(s-1))}\op C(r,\eta)\op C(r,\eta)_-$.

(12) Without loss of generality, we assume that $r\leq s$. It gives
 \[
\xymatrix@=1em{
\stt{U_1}\ar@{->}[rd]&&\stt{U_2}
\ar@{.>}[ld] \ar@{->}[rd] &\cdots&\stt{U_{r-1}}\ar@{->}[rd]&&\stt{U_r}\ar@{.>}[ld] \ar@{->}[rd]&\\
&\stt{V_1}&&\stt{V_2}&\cdots&\stt{V_{r-1}}&&\stt{V_r}},
\]
where $U_k=\sum_{i=1}^ku_i\ot u'_{k+1-i},V_k=\sum_{i=1}^k v_i\ot u'_{k+1-i}+u_i\ot v'_{k+1-i},~k=1,\dots,r$. Meanwhile,
\[
\xymatrix@=1em{
\stt{u_1\ot v'_s}\ar@{->}[rd]&&\stt{u_2\ot v'_s}\ar@{.>}[ld] \ar@{->}[rd] &\cdots&\stt{u_{r-1}\ot v'_s}\ar@{->}[rd]&&\stt{u_r\ot v'_s}\ar@{.>}[ld] \ar@{->}[rd]&\\
&\stt{v_1\ot v'_s}&&\stt{v_2\ot v'_s}&\cdots&\stt{v_{r-1}\ot v'_s}&&\stt{v_r\ot v'_s}}.\]

 \[\xymatrix@=1em{
&{\stt{u_i\ot u'_j}}\ar@{->}[ld] \ar@{.>}[rd]&\\
{\stt{v_i\ot u'_j+u_i\ot v'_j}}\ar@{.>}[rd]&&{\stt{v_{i-1}\ot u'_j+u_i\ot v'_{j-1}}}\ar@{->}[ld]^-{-1}\\
&{\stt{-v_i\ot v'_{j-1}+v_{i-1}\ot v'_j}}&}\]
for $i=1,\dots,r,j=2,\dots,s$, where we take $v_0=0$ for simplicity.

(13) It is similar to (12).

(14) For $N(r)\ot N'(s)$, we have
\[\xymatrix@=1em{
&{\stt{u_i\ot u'_j}}\ar@{->}[ld] \ar@{.>}[rd]&\\
{\stt{v_i\ot u'_j+u_i\ot v'_{j-1}}}\ar@{.>}[rd]&&{\stt{v_{i-1}\ot u'_j+u_i\ot v'_j}}\ar@{->}[ld]^-{-1}\\
&{\stt{-v_i\ot v'_j+v_{i-1}\ot v'_{j-1}}}&}\]
for $i=1,\dots,r,j=1,\dots,s$, where we take $v_0=v'_0=0$ for simplicity.

For $C(r,\eta)\ot N(s)$, we have
\[\xymatrix@=1em{
&{\stt{u_i\ot u'_j}}\ar@{->}[ld] \ar@{.>}[rd]&\\
{\stt{v_i\ot u'_j+u_i\ot v'_j}}\ar@{.>}[rd]&&{\stt{(\eta v_i+v_{i-1})\ot u'_j+u_i\ot v'_{j-1}}}\ar@{->}[ld]^-{-1}\\
&{\stt{-v_i\ot v'_{j-1}+(\eta v_i+v_{i-1})\ot v'_j}}&}\]
for $i=1,\dots,r,j=1,\dots,s$, where we take $v'_0=0$ for simplicity.

For $C(r,\eta)\ot N'(s)$, we have
\[\xymatrix@=1em{
&{\stt{u_i\ot u'_j}}\ar@{->}[ld] \ar@{.>}[rd]&\\
{\stt{v_i\ot u'_j+u_i\ot v'_{j-1}}}\ar@{.>}[rd]&&{\stt{(\eta v_i+v_{i-1})\ot u'_j+u_i\ot v'_j}}\ar@{->}[ld]^-{-1}\\
&{\stt{-v_i\ot v'_j+(\eta v_i+v_{i-1})\ot v'_{j-1}}}&}\]
for $i=1,\dots,r,j=1,\dots,s$, where we take $v'_0=0$ for simplicity.
\end{proof}

\section{The Green ring of $\bar{\ma}$ and its Jacobson radical}
\subsection{}
Let $H$ be a Hopf algebra. Recall that the \textit{Green} (or \textit{representation}) \textit{ring} $r(H)$ and
the \textit{Green algebra} $R(H)$ of $H$ can be defined as follows. $r(H)$ is the abelian group  generated by the
isomorphism classes $[M]$ of $M\in H\mb{-mod}$ modulo the relations $[M\op N] =[M]+[N]$.
The multiplication of $r(H)$ is given by
the tensor product of $H$-modules, that is, $[M][N]=[M\ot N]$. Then $r(H)$ is an associative ring with the identity $[\bK_\ve]$,
where $\bK_\ve$ is the trivial $H$-module. $R(H)$ is an associative $\bK$-algebra defined
by $\bK\ot_\mathbb{Z}r(H)$. Note that $r(H)$ is a free abelian group with a $\mathbb{Z}$-basis $\{\,[M]\mid M\in\mbox{ind}(H)\,\}$,
where $\mb{ind}(H)$ denotes the category of finite-dimensional indecomposable
$H$-modules. If $H$ is a quasitriangular Hopf algebra, then $M\ot N\cong N\ot M$ as $H$-modules.
In this case, $r(H)$ is a commutative ring.

 Combining Lemma \ref{sim} and Theorem \ref{dec}, we derive the Green ring $r(\bar{\ma})$ of $\bar{\ma}$.
\begin{cor}\label{green}
The Green ring $r(\bar{\ma}')$ of $\bar{\ma}'$ is a commutative ring generated by
\[ \Big\{\, [S(-,-)],\, [P],\, [M(1)],\, [W(1)]\, \Big\}\cup\Big\{\, [C(r,\eta)],\, [N(r)],\, [N'(r)]\, \Big\}_{r\in \mathbb{Z}^+,\eta\in\bK^\times},\]
subject to the following relations
\[\begin{array}{l}
S^2=1,\\
P^2=2(1{+}S)P,\quad MP=(1{+}2S)P,\quad WP=(2{+}S)P,\quad C_{r,\eta}P=N_rP=N'_rP=r(1{+}S)P,\\
MW=2P+S,\quad MC_{r,\eta}=rP+SC_{r,\eta},\quad MN_r=rP+SN_r,\quad MN'_r=rP+SN'_r,\\
WC_{r,\eta}=rP+C_{r,\eta},\quad WN_r=rP+N_r,\quad WN'_r=rP+N'_r,\\
C_{r,\eta}C_{s,\ga}=\begin{cases}
   rsP,&\eta\neq\ga,\\
   (rs-\mb{min}\{r,s\})P+(1{+}S)C_{\bs{min}\{r,s\},\eta},&\eta=\ga,
 \end{cases}\\
N_rN_s=(rs-\mbox{min}(r,s))P+(1{+}S)N_{\bs{min}(r,s)},\quad N'_rN'_s=(rs-\mbox{min}(r,s))P+(1{+}S)N'_{\bs{min}(r,s)},\\
C_{r,\eta}N_s=C_{r,\eta}N'_s=N_rN'_s=rsP,
\end{array}\]
where we abbreviate $[S(-,-)]$, $[P(+,+)]$, $[M(1)]$, $[W(1)]$, $[C(r,\eta)]$, $[N(r)]$,
$[N'(r)]$ to $S$, $P$, $M$, $W$, $C_{r,\eta}$, $N_r$, $N'_r$, successively. Moreover, the Green ring $r(\bar{\ma}')$ is just a subring of $r(\bar{\ma})$, which has one more generator $S_-:=[S(+,-)]$ and additional relations$:$ $S_-^2=1$, and
 \[[S,S_-]=[P,S_-]=[M,S_-]=[W,S_-]=[N_r,S_-]
 =[N'_r,S_-]=0,\,S_-C_{r,\eta}=C_{r,-\eta}S_-.\]
\end{cor}

\begin{rem}
In contrast to the commutativity of the Green ring of the generalized Taft algebra $H_{n,d}$ mentioned in \cite[Cor. 3.2]{LZ}, $r(\bar{\ma})$ is non-commutative.
\end{rem}

\subsection{}
 Recall that the \textit{projective class ring} $p(\bar{\ma})$ of $\bar{\ma}$ is the subring of $r(\bar{\ma})$
 generated by the projective modules and the simple modules (An interesting discussion see \cite{Cib2}). We consider the \textit{projective class algebra}
 $p(\bar{\ma})=\bK[S,S_-,P]/(S^2{-}1,S_-^2{-}1,P^2{-}2(1{+}S)P)$. The Jacobson radical $J(p(\bar{\ma}))=((1{-}S)P)$,
 thus \[p(\bar{\ma})/J(p(\bar{\ma}))=\bK[S,S_-,P]/
 (S^2-1,S_-^2-1,P^2-4P,(1-S)P)\cong\bK^6,\]
with the orthogonal idempotents $(1\pm S_-)(1-S)/4,\,(1\pm S_-)P/8,\,(1\pm S_-)(1-\tfrac{P}{4}-\tfrac{1-S}{2})/2$.

Note that all projective modules in $r(\bar{\ma})$ generate an ideal $\mathcal {P}$,
and one can define the \textit{stable Green ring} of $\bar{\ma}$ as $\mb{St}(\bar{\ma})=r(\bar{\ma})/\mathcal {P}$. The stable Green ring was introduced in the study of Green rings for the modular representation theory of finite groups (see \cite{Ben}).  In order to compute the Jacobson radicals of such rings, we need the following Lemma.
\begin{lem}[{\cite[Chapter 2, 4.3]{Lam}}]\label{jac}
For an arbitrary ring $R$ and $y\in R$, the following statements are equivalent:

(1) $y\in J(R)$;

(2) $1-xyz\in U(R)$ (the group of units of $R$) for any $x,z\in R$.
\end{lem}

Now for the Green algebras $R(\bar{\ma}'),R(\bar{\ma})$, we have
\begin{theorem}\label{rad}
For the quotient $\bar{\ma}'$ of $\bar{\ma}$, the Jacobson radicals of $\mb{St}(\bar{\ma}')$, $R(\bar{\ma}')$ are respectively equal to
\begin{gather*}
J(\mb{St}(\bar{\ma}'))=\lb(S{-}1)N_r,\, (S{-}1)N'_r,\, (S{-}1)C_{r,\eta}\;\Bigl|\;r\in\bZ^+,\eta\in\bK^\times\rb\Bigr.,\\
J(R(\bar{\ma}'))=\lb(S{-}1)P,\, (S{-}1)N_r,\, (S{-}1)N'_r,\, (S{-}1)C_{r,\eta}\;\Bigl|\;r\in\bZ^+,\eta\in\bK^\times\rb\Bigr..
\end{gather*}
Furthermore, $J(\mb{St}(\bar{\ma})),\,J(R(\bar{\ma}))$ have the same presentations as above.
\end{theorem}
\begin{proof}
We first deal with $\mb{St}(\bar{\ma}')$. Take $J=\big((S{-}1)N_r,(S{-}1)N'_r,(S{-}1)C_{r,\eta}\mid r\in\bZ^+,\eta\in\bK^\times\big)$
and $A=\mb{St}(\bar{\ma}')/J$. Clearly, $J$ is nilpotent, we only need to show that $A$ is Jacobson semisimple.
First one can check that
\[\big\{\,(N_{r+1}{-}N_r)/2,\,(N'_{r+1}{-}N'_r)/2,\,
(C_{r+1,\eta}{-}C_{r,\eta})/2\,\big|\,r\in\bN,\eta\in\bK^\times\,\big\}\]
is a set of orthogonal idempotents in $A$, where we set $N_0=N'_0=C_{0,\eta}=0$ for simplicity. Denote by $I$ the ideal of $A$ generated by the set of such idempotents. These idempotents also form a $\bK$-basis of $I$. Let $\bar{A}:=A/I$, then $\bar{A}$ is generated by $S,M,W$ and isomorphic to $\bK[\bZ\times\bZ_2]$ as $S^2=1,\,SMW=1$. $\bar{A}$ is clearly Jacobson semisimple and thus $J(A)\subseteq I$. Now let $x=\sum_{r\in\bN}a_r(N_{r+1}-N_r)/2+b_r(N'_{r+1}-N'_r)/2+
\sum_{r\in\bN,\eta\in\bK^\times}c_{r,\eta}(C_{r+1,\eta}-C_{r,\eta})/2$ be an arbitrary element in $I$, where $\{a_r,b_r,c_{r,\eta}\in\bK,\,r\in\bN,\eta\in\bK^\times\}$ only has finite non-zero elements. If $x\neq0$, then we can assume that there exists some non-zero $a_r$ without loss of generality such that $\lb1-a_r^{-1}x\rb(N_{r+1}-N_r)=0$. It means that $1-a_r^{-1}x$ is a non-zero divisor and not invertible, hence by Lemma \ref{jac}
$x\notin J(A)$, i.e. $J(A)=J(A)\cap I=0$.

Next we deal with the case of $R(\bar{\ma}')$. From the short exact sequence $0\rw\mathcal {P}\rw R(\bar{\ma}')\stackrel{\pi}{\rw}\mb{St}(\bar{\ma}')\rw0$,
we see that
\[J(R(\bar{\ma}'))\subset \pi^{-1}(J(\mb{St}(\bar{\ma}')))=\lb P,\,(S{-}1)N_r,\,(S{-}1)N'_r,\,(S{-}1)C_{r,\eta}\;\big|\;r\in\bZ^+,\eta\in\bK^\times\rb.\]
Meanwhile, since $(S{-}1)P, (S{-}1)N_r, (S{-}1)N'_r, (S{-}1)C_{r,\eta}\; ~(r\in\bZ^+,\eta\in\bK^\times)$ are all nilpotent in $R(\bar{\ma}')$,
they all lie in $J(R(\bar{\ma}'))$. In order to see that they generate $J(R(\bar{\ma}'))$,
one only needs to check that $(S{+}1)P/8$ is an idempotent and the ideal $((S{+}1)P)=\bK(S{+}1)P$, thus $((S{+}1)P)\cap J(R(\bar{\ma}'))=0$.

The case of $R(\bar{\ma})$ is a little bit more complicated. Analogously, we set \[J:=\lb(S{-}1)N_r,\,(S{-}1)N'_r,\,(S{-}1)C_{r,\eta}
\;\big|\;r\in\bZ^+,\eta\in\bK^\times\rb\]
as a nilpotent ideal of $\mb{St}(\bar{\ma})$ and $A:=\mb{St}(\bar{\ma})/J$.
Also, set $\bar{A}:=A/I$, where
\[I:=\lb(N_{r+1}-N_r)/2,\,(N'_{r+1}-N'_r)/2,
\,(C_{r+1,\eta}-C_{r,\eta})/2\,\big|\,r\in\bN,\eta\in\bK^\times\rb.\]
Then $\bar{A}$ is generated by $\{S,S_-,M,W\}$ and isomorphic to $\bK[\bZ_2\times\bZ_2\times\bZ]$, thus Jacobson semisimple.
Again we only need to show that $J(A)=J(A)\cap I=0$.
Now note that
\[
\Big\{\,(1\pm S_-)(N_{r+1}-N_r),\,(1\pm S_-)(N'_{r+1}-N'_r),\,(1\pm S_-)(C_{r+1,\eta}-C_{r,\eta})\,\big|\,r\in\bN,\eta\in\bK^\times\,\Big\}
\]
forms a $\bK$-basis of $I$. let $x:=\sum_{r\in\bN}(a_r(1+S_-)+a'_r(1-S_-))(N_{r+1}-N_r)
+(b_r(1+S_-)+b'_r(1-S_-))(N'_{r+1}-N'_r)+
\sum_{r\in\bN,\eta\in\bK^\times}
(c_{r,\eta}(1+S_-)+c'_{r,\eta}(1-S_-))(C_{r+1,\eta}-C_{r,\eta})$ be an arbitrary element in $I$ as before. If $x\neq0$,
then we may assume that there exists some non-zero $c_{r,\eta}$ without loss of generality. It is easy to check that
\[\begin{array}{l}
(1+S_-)x(C_{r+1,\eta}-C_{r,\eta})=4c_{r,\eta}(1+S_-)(C_{r+1,\eta}-C_{r,\eta}),\\
\lb(1+S_-)(C_{r+1,\eta}-C_{r,\eta})\rb^2
=2(1+S_-)(C_{r+1,\eta}-C_{r,\eta}).
\end{array}\]
Hence,
\[\lb1-(8c_{r,\eta})^{-1}(1+S_-)x(C_{r+1,\eta}-C_{r,\eta})\rb
(1+S_-)(C_{r+1,\eta}-C_{r,\eta})=0\]
and $1-(8c_{r,\eta})^{-1}(1+S_-)x(C_{r+1,\eta}-C_{r,\eta})$ is not invertible, thus $x\notin J(A)$, i.e. $J(A)=0$, by Lemma \ref{jac} again. Now in order to apply such result to the case of $R(\bar{\ma})$, one only needs to see that $p_\pm:=(1\pm S_-)(S{+}1)P/16$ are orthogonal idempotents and the ideal $(p_\pm)=\mb{span}_\bK\{p_\pm\}$, thus $(p_\pm)\cap J(R(\bar{\ma}))=0$.
\end{proof}

\section{The Green rings of two $2$-cocycle twists of $\bar{\ma}$}

In this section, we will discuss two $2$-cocycle twists of $\bar{\ma}$, both of which involve the $4$-dimensional
Sweedler algebra $H_4$. $H_4=\bK\lan a, b\ran/(a^2-1, b^2, ab+ba)$ is the simplest example of
noncommutative and noncocommuative Hopf algebra with $S^2\ne 1$.

\subsection{Two $2$-cocycle twists of $\bar{\ma}$, $H_4\ot H_4$ and $D(H_4)$} First of all,
let us briefly recall the definition of $2$-cocycle twist of a Hopf algebra (see \cite{dt}).
Associated with a $2$-cocycle $\si$ as a bilinear form defined on a bialgebra $H$,
which is invertible under the convolution product and satisfies
\begin{gather*}
\si(a,1)=\si(1,a)=\ve(a),\quad a\in H,\\
\si(a_{(1)},b_{(1)})\,\si(a_{(2)}b_{(2)},c)=\si(b_{(1)},c_{(1)})\,\si(a,b_{(2)}c_{(2)}),\quad a,\, b,\, c\in H,
\end{gather*}
one can construct a new bialgebra $(H^\si,\cdot_\si,\De,\ve)$ with
\[a\cdot_\si b=\si(a_{(1)},b_{(1)})\,a_{(2)}b_{(2)}\,\si^{-1}(a_{(3)},b_{(3)}),\quad a,\, b\in H.\]
Moreover, if $H$ is a Hopf algebra with the antipode $S$, then the antipode $S^\si$ of $H^\si$ is given by
$S^\si(a)=\lan \msu,a_{(1)}\ran\, S(a_{(2)})\,\lan \msu^{-1},a_{(3)}\ran$, for $a\in H^\si$,
where $\msu=\si\circ(\mi\ot S)\circ\De\in H^*$ with the inverse $\msu^{-1}=\si^{-1}\circ(S\ot\mi)\circ\De$.

The first twist involved is just the tensor product $H_4\ot H_4$ of the Sweedler Hopf algebra $H_4$,
with the corresponding $2$-cocycle $\si_1$ on $\bar{\ma}$ given by:
\[\si_1(u,v)=
\begin{cases}
(-1)^{a_1b_2},&u=g^{a_1}h^{a_2},\; v=g^{b_1}h^{b_2},\\
0,&\mb{$u$, or $v\notin \bK\lan g, h\ran$,}
\end{cases}\]
where $a_1,\, a_2,\, b_1, \,b_2\in\bZ$. From \cite[p.294]{Rad}, we know that $H_4$ is quasitriangular,
whose universal $R$-matrix is given by $R=\tfrac{1}{2}(1\ot1+g\ot1+1\ot g-g\ot g)$. Hence,
$H_4\ot H_4$ is also quasitriangular.

The other twisted we concern is the Drinfel'd double $D(H_4)$, whose Green ring has been computed in \cite{chen1}.
Here we point out that it is also $2$-cocycle twist-equivalent to $H_4\ot H_4$, thus to $\bar{\ma}$.
Now we recall the presentation of $D(H_4)$ given in \cite{chen1}.
It has generators $g,\, h,\, x,\, y$, which subject to the following relations:
\[\begin{array}{l}
g^2=h^2=1,\quad x^2=y^2=0,\quad gx=-xg,\quad gy=-yg,\\
hx=-xh,\quad hy=-yh,\quad xy+yx=1-gh.
\end{array}\]
Meanwhile, we define the comultiplication of $D(H_4)$ to be the same as $\bar{\ma}$, opposite to the one used in \cite{chen1}.
Due to \cite{dt}, for two given Hopf algebras $A, B$, if there exists a skew pairing $\lan,\ran: B\ot A\rw\bK$, satisfying
\[\begin{array}{l}
\lan bb',a\ran=\lan b\ot b',\De(a)\ran,\quad \lan b,aa'\ran=\lan\De^\bs{op}(b),a\ot a'\ran,\\
\lan 1,a\ran=\ve(a),\quad \lan b,1\ran=\ve(b),
\end{array}\]
then one can define a $2$-cocycle $\si$ on $A\ot B$ by
\[\si(a\ot b, a'\ot b')=\ve(a)\lan b, a'\ran\ve(b),\quad a, \,a'\in A, \;b, \,b'\in B,\]
such that the twist $(A\ot B)^\si$ is again a Hopf algebra. Now as $\{b^ia^j\}_{0\leq i,j\leq1}$
forms a basis of $H_4$, we can define a skew pairing $\lan,\ran:H_4\ot H_4\rw\bK$ by
\[\lan a^i, a^j\ran=(-1)^{ij},\quad \lan b, b\ran=1,\quad \lan a^i, b\ran=\lan b, a^j\ran=0\]
with the corresponding $2$-cocycle $\si_2$.

\begin{prop}
There exists a Hopf algebra isomorphism $\phi$ between $(H_4\ot H_4)^{\si_2}$ and $D(H_4)$ defined by $\phi(b^ia^j\ot b^ka^l)=x^ig^jy^kh^l$, for $\,0\leq i,\, j,\, k,\, l\leq1$.
\end{prop}
\begin{proof}
It is straightforward to check that
\[\begin{array}{l}
\phi((1\ot b)\cdot_{\si_2}(b\ot 1))=1-gh-xy=yx=\phi(1\ot b)\phi(b\ot 1),\\
\phi((1\ot b)\cdot_{\si_2}(a\ot 1))=-gy=yg=\phi(1\ot b)\phi(a\ot 1),\\
\phi((1\ot a)\cdot_{\si_2}(b\ot 1))=-xh=hx=\phi(1\ot a)\phi(b\ot 1).
\end{array}\]
This is a Hopf algebra isomorphism.\end{proof}

\subsection{The Green ring of $D(H_4)$}
Motivated by the construction of a complete set of primitive orthogonal idempotents of $\bar{U}_{r,s}(sl_2)$ in \cite{TH},
we figure out those of $D(H_4)$.
\begin{prop}
A complete set of primitive orthogonal idempotents of $D(H_4)$ is given by
$\{e_1,\dots,e_6\}$, where
\[\begin{array}{l}
e_1=\tfrac{1}{4}(1+g+h+gh),\qquad e_2=\tfrac{1}{4}(1-g-h+gh),\\
e_3=\tfrac{1}{8}xy(1+g-h-gh),\quad e_4=\tfrac{1}{8}(2-xy)(1+g-h-gh),\\
e_5=\tfrac{1}{8}xy(1-g+h-gh),\quad e_6=\tfrac{1}{8}(2-xy)(1-g+h-gh).
\end{array}\]
\end{prop}
\begin{proof}
It is straightforward to check.
\end{proof}

\begin{rem}
In \cite{Ari}, Arike described a complete set of primitive orthogonal idempotents of $\bar{U}_q(sl_2)$ with $q$ a primitive $2p$-th root of unity, $p\geq2$.
Later, Kondo, Saito \cite{KS} decomposed the indecomposable decomposition of tensor products of modules over $\bar{U}_q(sl_2)$ based on Arike's work.
However, the truncation relations in their definition of $\bar{U}_q(sl_2)$ are distinct from $D(H_4)$ in \cite{chen1}.
From this point of view, the double $D(H_4)$ considered by Chen does not belong to those small quantum groups studied by Kondo and Saito in \cite{KS}.
Hence, we prefer to use the version (close to ours) of restricted two-parameter quantum groups in \cite{TH}.
\end{rem}

Note that $D(H_4)$ is not basic, as $D(H_4)e_3\cong D(H_4)e_4,~D(H_4)e_5\cong D(H_4)e_6$.
Each of them is a two dimensional simple projective module. We denote $P_+=D(H_4)e_3,~P_-=D(H_4)e_5$,
whose diagrams are respectively
\[\xymatrix@=2em{\stt{e_3}\ar@{.>}@/^/[d]\\
\stt{ye_3}\ar@{->}@/^/[u]^-{2}},\quad\xymatrix@=2em{\stt{e_5}\ar@{.>}@/^/[d]\\
\stt{ye_5}\ar@{->}@/^/[u]^-{2}}.\]
Meanwhile, the radical $J(D(H_4))=(x(1+gh),y(1+gh))$.
Hence, the Gabriel quiver of the basic subalgebra $D(H_4)e_1\op D(H_4)e_2\op e_3D(H_4)e_3\op e_5D(H_4)e_5$ of $D(H_4)$ looks like
\[\xymatrix@=2em{\xy 0;/r.1pc/:
(0,0)*{\circ};(-6,0)*{\stt{e_3}};\endxy&
\xy 0;/r.1pc/:(0,0)*{\circ};
(-6,0)*{\stt{e_1}};\endxy
\ar@{->}@/^/[r]^-{\be_1}\ar@{->}@/^{1.7em}/[r]^-{\al_1}&
\xy0;/r.1pc/:(0,0)*{\circ};(6,0)*{\stt{e_2}};\endxy
\ar@{->}@/^/[l]^-{\be_2}\ar@{->}@/^{1.7em}/[l]^-{\al_2}
&\xy 0;/r.1pc/:(0,0)*{\circ};
(6,0)*{\stt{e_5}};\endxy
}\]
Note that $D(H_4)$-mod has three blocks, the modules from the block $D(H_4)e_1\op D(H_4)e_2$-mod are
the same as those from $\bar{\ma}f_1$-mod, on which the action of $1-gh$ vanishes.
So all the tensor product decompositions from $\bar{\ma}f_1$-mod
in Theorem \ref{dec} work on $D(H_4)e_1\op D(H_4)e_2$-mod as well.

Now due to Corollary \ref{green}, the Green ring of $D(H_4)$ can be obtained directly.
One should compare our list with the result in \cite{chen1},
which is described in a different language, using syzygy, cosyzygy functors, etc.

\begin{cor}\label{gre2}
The Green ring $r(D(H_4))$ of $D(H_4)$ is a commutative ring generated by
\[\Big\{\,[S(-,-)], \,[P], \,[P_+], \,[M(1)], \,[W(1)]\,\Big\}\cup\Big\{\,[C(r,\eta)],\, [N(r)],\, [N'(r)]\,\Big \}_{r\in \mathbb{Z}^+,\eta\in\bK^\times},\]
subject to the following relations
\[\begin{array}{l}
S^2=1,\quad P_+P=2(1{+}S)P_+,\quad P_+^2=SP,\\
MP_+=(1+2S)P_+,\quad WP_+=(2{+}S)P_+,\quad C_{r,\eta}P_+=N_rP_+=N'_rP_+=r(1{+}S)P_+,\\
MW=2P+S,\quad MC_{r,\eta}=rP+SC_{r,\eta},\quad MN_r=rP+SN_r,\quad MN'_r=rP+SN'_r,\\
WC_{r,\eta}=rP+C_{r,\eta},\quad WN_r=rP+N_r,\quad WN'_r=rP+N'_r,\\
C_{r,\eta}C_{s,\ga}=\begin{cases}
   rsP,&\eta\neq\ga,\\
   (rs-\mb{min}\{r,s\})P+(1{+}S)C_{\bs{min}\{r,s\},\eta},&\eta=\ga,
 \end{cases}\\
N_rN_s=(rs-\mbox{min}(r,s))P+(1{+}S)N_{\bs{min}(r,s)},
\quad N'_rN'_s=(rs-\mbox{min}(r,s))P+(1{+}S)N'_{\bs{min}(r,s)},\\
C_{r,\eta}N_s=C_{r,\eta}N'_s=N_rN'_s=rsP,
\end{array}\]
where $S,\, P,\, M,\, W,\, C_{r,\eta},\, N_r,\, N'_r$ are still for $[S(-,-)]$, $[P(+,+)]$, $[M(1)]$, $[W(1)]$, $[C(r,\eta)]$, $[N(r)]$, $[N'(r)]$, successively.
\end{cor}

The projective class algebra $p(D(H_4))=\bK[S,P_+]/(S^2-1,P_+^3-2(1{+}S)P_+)$.
The Jacobson radical $J(p(D(H_4)))=((1{-}S)P_+)$,
thus \[p(D(H_4))/J(p(D(H_4)))=\bK[S,P_+]/(S^2-1,P_+^3-4P_+,(1{-}S)P_+)
\cong\bK^4,\]
with the orthogonal idempotents $(1-S)/2,\,P_+(P_+\pm2)/8,\,\tfrac{1+S}{2}-\tfrac{P_+^2}{4}$.
Meanwhile, using Theorem \ref{rad} we get that
\begin{theorem}
The Jacobson radical
\[J(R(D(H_4)))=\lb(1{-}S)P_+,\,(1{-}S)N_r,\, (1{-}S)N'_r,\, (1{-}S)C_{r,\eta}\;\big|\;r\in\bZ^+,\eta\in\bK^\times\rb.\]
\end{theorem}
\begin{proof}
First it is clear that $\mb{St}(D(H_4))=\mb{St}(\bar{\ma}')$. From Theorem \ref{rad}, we know that
\[J(R(D(H_4)))\subset\lb P_+,\,SP_+,\,(1{-}S)N_r,\, (1{-}S)N'_r,\,(1{-}S)C_{r,\eta}
\;\big|\;r\in\bZ^+,\eta\in\bK^\times\rb.\]

In order to derive $J(R(D(H_4)))$, one only needs to see that
$$((1{+}S)P_+)=\mb{span}_\bK\Big\{\,(1{+}S)(P_+{\pm}2)P_+/16\,\Big\}.$$
On the other hand, $(1{+}S)(P_+{\pm}2)P_+/16$ are orthogonal idempotents, thus $((1{+}S)P_+)\cap J(R(D(H_4)))=0$.
\end{proof}

\subsection{The Green ring of $\mh$} Write $\mh:=H_4\ot H_4$ for short.
The rest of the paper will be devoted to computing the Green ring of
$\mh$. We use the following presentation of $\mh$. It has generators
$g, h, x, y$ as well, but subject to the following different
relations,
\begin{equation}\label{rel}
\begin{array}{l}
gh=hg,\quad g^2=h^2=1,\\
gx=-xg,\quad gy=yg,\quad hx=xh,\quad hy=-yh,\\
x^2=y^2=0,\quad xy=yx.
\end{array}
\end{equation}
The comultiplication $\De$ is defined by
\[\De(g)=g\ot g,\quad \De(h)=h\ot h,\quad \De(x)=x\ot 1+g\ot x,\quad \De(y)=y\ot 1+h\ot y.\]
$\mh$ also has $4$ orthogonal primitive idempotents \[e_1=\tfrac{1}{4}(1+g+h+gh),e_2=\tfrac{1}{4}(1+g-h-gh),
e_3=\tfrac{1}{4}(1-g+h-gh),e_4=\tfrac{1}{4}(1-g-h+gh),\]
but the unit $1$ is the unique central idempotent of $\mh$.

\begin{lem}
(1) There exist two signs $s_1,s_2\in\{+,-\}$ such that $ge_i=s_1e_i,\; he_i=s_2e_i$ with $(s_1,s_2)=(+,+),\, (+,-),\, (-,+),\, (--)$ for $i=1, 2, 3, 4$, successively.

(2) $xe_1=e_3x,\quad xe_3=e_1x,\quad ye_1=e_2y,\quad ye_2=e_1y,\quad xe_2=e_4x,\quad xe_4=e_2x$,

\hskip0.6cm $ye_4=e_3y,\quad ye_3=e_4y$.
\end{lem}

Similar to $\bar{\ma}$, we denote by $S(s_1,s_2)$ the one-dimensional simple module $\bK$ with $g\cdot1=s_11,h\cdot1=s_21,x\cdot1=y\cdot1=0$, and $P(s_1,s_2)$ as the projective cover of $S(s_1,s_2)$.

As $P(s_1,s_2)$ are non-isomorphic to each other with respect to the signs $(s_1,\, s_2)$, $\mh$ is also a basic algebra over $\bK$.
On the other hand, the radical $J(\mh)=(x,y)$, thus from the above lemma, we know that the Gabriel quiver $Q_{\mh}$ of $\mh$ looks like:
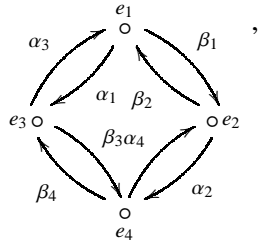
\[\xymatrix@=1.7em{&
\xy 0;/r.1pc/:
(0,0)*{\circ};(0,6)*{\stt{e_1}};\endxy\ar@{->}@/^/[rd]^-{\be_1}\ar@{->}@/^/[ld]^-{\al_1}&\\
\xy 0;/r.1pc/:
(0,0)*{\circ};(-6,0)*{\stt{e_3}};\endxy\ar@{->}@/^/[ru]^-{\al_3}\ar@{->}@/^/[rd]^-{\be_3}
&&\xy 0;/r.1pc/:
(0,0)*{\circ};(6,0)*{\stt{e_2}};\endxy\ar@{->}@/^/[lu]^-{\be_2}\ar@{->}@/^/[ld]^-{\al_2}\\
&\xy 0;/r.1pc/:
(0,0)*{\circ};(0,-6)*{\stt{e_4}};\endxy\ar@{->}@/^/[ru]^-{\al_4}\ar@{->}@/^/[lu]^-{\be_4}&}
,\]
where for $i=1, 2, 3, 4$, the arrows $\al_i,\, \be_i$ correspond to $xe_i,\, ye_i$, respectively.
The admissible ideal $I$ has the following relations:
\[\begin{array}{l}
\al_1\al_3=\al_3\al_1=0,\quad \al_2\al_4=\al_4\al_2=0,\quad
\be_1\be_2=\be_2\be_1=0,\quad \be_3\be_4=\be_4\be_3=0,\\
\be_3\al_1-\al_2\be_1=\be_4\al_2-\al_1\be_2=
\be_1\al_3-\al_4\be_3=\be_2\al_4-\al_3\be_4=0.
\end{array}
\]
Hence, $\mh$ has a unique block.

From the quiver $Q_\mh$, we know that $\mh$ is also a special biserial algebra.
We describe all its indecomposable modules of Loewy length $2$ as follows.
For any $r,s\in\mathbb{Z}^+,\eta,\ga\in \bK^\times$, we first define the following string modules
\[\begin{array}{l}
M(r)=\begin{cases}
M\lb(\al_1\be_1^{-1}\al_4\be_4^{-1})^{r/2}\rb,&\mb{if }r\mb{ is even},\\
M\lb(\al_1\be_1^{-1}\al_4\be_4^{-1})^{(r-1)/2}\al_1\be_1^{-1}\rb,&\mb{if }r\mb{ is odd},
\end{cases}\\
W(r)=\begin{cases}
M\lb(\be_1^{-1}\al_4\be_4^{-1}\al_1)^{r/2}\rb,&\mb{if }r\mb{ is even},\\
M\lb(\be_1^{-1}\al_4\be_4^{-1}\al_1)^{(r-1)/2}\be_1^{-1}\al_4\rb,&\mb{if }r\mb{ is odd},
\end{cases}\\
\mb{all of dimension }2r+1.\\
N(r)=\begin{cases}
M\lb(\al_1\be_1^{-1}\al_4\be_4^{-1})^{(r-2)/2}\al_1\be_1^{-1}\al_4\rb,&\mb{if }r\mb{ is even},\\
M\lb(\al_1\be_1^{-1}\al_4\be_4^{-1})^{(r-1)/2}\al_1\rb,&\mb{if }r\mb{ is odd},
\end{cases}\\
N'(r)=\begin{cases}
M\lb(\be_1^{-1}\al_4\be_4^{-1}\al_1)^{(r-2)/2}\be_1^{-1}\al_4\be_4^{-1}\rb,&\mb{if }r\mb{ is even},\\
M\lb(\be_1^{-1}\al_4\be_4^{-1}\al_1)^{(r-1)/2}\be_1^{-1}\rb,&\mb{if }r\mb{ is odd},
\end{cases}\\
\mb{all of dimension }2r.
\end{array}
\]
The diagrams of $M(r),\, N(r)$ look like:
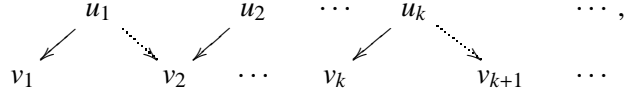
\[\xymatrix@=1em{
&u_1\ar@{->}[ld] \ar@{.>}[rd]& &u_2\ar@{->}[ld] &\cdots&u_k\ar@{->}[ld] \ar@{.>}[rd]&&\cdots\\
v_1&&v_2&\cdots&v_k&&v_{k+1}&\cdots},\]
while the diagrams of $W(r),\, N'(r)$ look like:
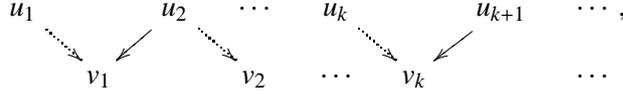
\[\xymatrix@=1em{
u_1\ar@{.>}[rd]& &u_2\ar@{->}[ld] \ar@{.>}[rd] &\cdots&u_k\ar@{.>}[rd]&&u_{k+1}\ar@{->}[ld]&\cdots\\
&v_1&&v_2&\cdots&v_k&&\cdots},\]
where $u_{1}$ always has the sign $(+,+)$, and we again use the arrow $\xy0;/r.15pc/:{\ar@{->}(-5,0)*{};(5,0)*{}};\endxy$ (resp. $\xy0;/r.15pc/:{\ar@{.>}(-5,0)*{};(5,0)*{}};\endxy$)
to represent the action of $x$ (resp. $y$) on the modules.

Meanwhile, we define the following band modules
\[C(r,\eta)=M\lb\be_4\al_4^{-1}\be_1\al_1^{-1},r,\eta\rb\]
of dimension $4r$, whose diagrams are as follows:
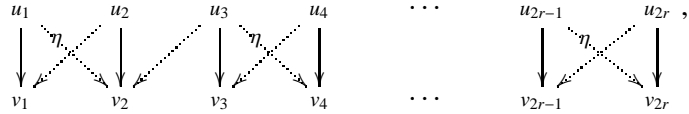
\[\xymatrix@=2em{
\stt{u_1}\ar@{->}[d] \ar@{.>}[rd]|<<<\eta&\stt{u_2}\ar@{->}[d] \ar@{.>}[ld] &\stt{u_3}\ar@{->}[d]\ar@{.>}[rd]|<<<\eta\ar@{.>}[ld]
&\stt{u_4}\ar@{->}[d] \ar@{.>}[ld]&\cdots&\stt{u_{2r-1}}\ar@{->}[d] \ar@{.>}[rd]|<<<\eta&\stt{u_{2r}}\ar@{->}[d] \ar@{.>}[ld]\\
\stt{v_1}&\stt{v_2}&\stt{v_3}&\stt{v_4}&\cdots&\stt{v_{2r-1}}&\stt{v_{2r}}},\]
where $y\cdot u_{2i-1}=\eta v_{2i}+v_{2i-2},~i=1,\dots,r$ and $v_0=0$.
Moreover, $u_{2i-1},~u_{2i}$ have the signs $(-,-),~(+,+)$, respectively.

Let $S(+,+)\triangleq1,\,P(+,+)\triangleq P$ and denote $S(s_1,s_2)$ by $S_{s_1s_2}$ when $(s_1,s_2)\neq(+,+)$.
For $M\in\{P,M(r),W(r),N(r),N'(r),C(r,\eta)\}$, we define $M_{s_1s_2}=M\ot S(s_1,s_2)$ and may omit $++$ for short.
Now after adding the signs, we get all the indecomposable modules of Loewy length $2$,
 consisting of the string ones and the band ones.
Moreover, it should be seen that \[\big\{\,S_{s_1,s_2},P_{s_1,s_2},M(r)_{s_1,s_2},W(r)_{s_1,s_2},
N(r)_{s_1,s_2},N'(r)_{s_1,s_2},C(r,\eta),C(r,\eta)_{+,-}
\,|\,s_i=\pm\,\big\}_{r\in\bZ^+,\eta\in \bK^\times}\]
gives the complete list of indecomposable modules of $\mh$.

Now for any $t\in\bQ$, we define $\lc t\rc$ (resp. $\lf t\rf$) as the lower (resp. upper) bound of
integers bigger (resp. smaller) than $t$. For any $r,s\in\bZ^+$, we write $\mb{min}\{r,s\},~\mb{max}\{r,s\}$ as $\ull{r,s},~\ol{r,s}$,
respectively for short. Moreover, we define the parity \[|r|=\begin{cases}
-,&\mb{if }r\mb{ is odd},\\
+,&\mb{if }r\mb{ is even}.
\end{cases}\]

\begin{theorem}\label{dec1}
For any $r,s\in\bZ^+,\eta,\ga\in \bK^\times$, we have

(1) $P\ot P\cong P\op P_{--}\op P_{+-}\op P_{-+}$.

(2) $\begin{array}{l}
M(r)\ot P\cong
P^{\op\lc r/2\rc}\op P_{--}^{\op\lf r/2\rf}\op P_{+-}^{\op\lc r/2\rc}\op P_{-+}^{\op \lc(r+1)/2\rc}.\\
W(r)\ot P\cong P^{\op\lc(r+1)/2\rc}\op P_{--}^{\op\lc r/2\rc}\op P_{+-}^{\op\lc r/2\rc}\op P_{-+}^{\op\lf r/2\rf}.
\end{array}$

(3) $\begin{array}{l}
N(r)\ot P\cong P^{\op\lc r/2\rc}\op P_{--}^{\op\lf r/2\rf}\op P_{+-}^{\op\lf r/2\rf}\op P_{-+}^{\op\lc r/2\rc}.\\
N'(r)\ot P\cong P^{\op\lc r/2\rc}\op P_{--}^{\op\lf r/2\rf}\op P_{+-}^{\op\lc r/2\rc}\op P_{-+}^{\op\lf r/2\rf}.\\
C(r,\eta)\ot P\cong P^{\op r}\op P_{--}^{\op r}\op P_{+-}^{\op r}\op P_{-+}^{\op r},~C(r,\eta)\ot S_{--}\cong C(r,\eta).
\end{array}$

(4) $M(r)\ot M(s)\cong P^{\op\lc rs/2\rc}\op P_{--}^{\op\lf rs/2\rf}\op M(r{+}s)_{-+}$.

(5) $W(r)\ot W(s)\cong P^{\op\lf rs/2\rf}\op P_{--}^{\op\lc rs/2\rc}\op W(r{+}s)$.

(6) $M(r)\ot W(s)\cong\begin{cases}
P^{\op\lc r(s+1)/2\rc}\op P_{--}^{\op\lf r(s+1)/2\rf}\op W(s{-}r)_{|r-1|,|r|},& r<s,\\
P^{\op r(r+1)/2}\op P_{--}^{\op r(r+1)/2}\op S_{|r-1|,|r|},& r=s,\\
P^{\op\lc (r+1)s/2\rc}\op P_{--}^{\op\lf (r+1)s/2\rf}\op M(r{-}s)_{|s|,|s|},& r>s.
\end{cases}$

(7) $M(r)\ot C(s,\eta)\cong P^{\op rs}\op P_{--}^{\op rs}\op C(s,\eta)_{+-}$.

(8) $W(r)\ot C(s,\eta)\cong P^{\op rs}\op P_{--}^{\op rs}\op C(s,\eta)$.

(9) $\begin{array}{l}
M(r)\ot N(s)\cong P^{\op\lc rs/2\rc}\op P_{--}^{\op\lf rs/2\rf}\op N(s)_{-+}.\\
M(r)\ot N'(s)\cong P^{\op\lc rs/2\rc}\op P_{--}^{\op\lf rs/2\rf}\op N'(s)_{+-}.
\end{array}$

(10) $\begin{array}{l}
W(r)\ot N(s)\cong P^{\op\lc rs/2\rc}\op P_{--}^{\op\lf rs/2\rf}\op N(s)_{|r|,|r|}.\\
W(r)\ot N'(s)\cong P^{\op\lf rs/2\rf}\op P_{--}^{\op\lc rs/2\rc}\op N'(s).
\end{array}$.

(11) $C(r,\eta)\ot C(s,\ga)\cong
 \begin{cases}
   P^{\op 2rs}\ot P_{--}^{\op 2rs},&\eta\neq\ga,\\
   P^{\op(2rs-\ull{r,s})}\op P_{--}^{\op(2rs-\ull{r,s})}\op C(\ull{r,s},\eta)\op C(\ull{r,s},\eta)_{+-},&\eta=\ga.
 \end{cases}$

(12) $N(r)\ot N(s)\cong P^{\op\lc(rs-\ull{r,s})/2\rc}\op P_{--}^{\op\lf(rs-\ull{r,s})/2\rf}\op N(\ull{r,s})_{|\ol{r,s}-1|,|\ol{r,s}-1|}\op
N(\ull{r,s})_{-+}$.

(13) $N'(r)\ot N'(s)\cong P^{\op\lf(rs-\ull{r,s})/2\rf}\op P_{--}^{\op\lc(rs-\ull{r,s})/2\rc}\op N'(\ull{r,s})\op
N'(\ull{r,s})_{|\ol{r,s}-1|,|\ol{r,s}|}$.

(14) $\begin{array}{l}
N(r)\ot N'(s)\cong P^{\op\lc rs/2\rc}\op P_{--}^{\op\lf rs/2\rf}.\\
C(r,\eta)\ot N(s)\cong C(r,\eta)\ot N'(s)\cong P^{\op rs}\op P_{--}^{\op rs}.
\end{array}$
\end{theorem}

\begin{proof}
(Sketch) For (1), (2), (3), note that the decomposition of the tensor product of $M$ with projective modules only depends on the composition factors of $M$.
On the other hand, by the construction of band modules, $M\lb\be_4\al_4^{-1}\be_1\al_1^{-1},r,\eta\rb\cong M\lb\be_1\al_1^{-1}\be_4\al_4^{-1},r,\eta\rb$.
That is, $C(r,\eta)\ot S_{--}\cong C(r,\eta)$. Meanwhile, $C(r,\eta)\ot S_{+-}$ corresponds to the band module $M\lb\al_3\be_3^{-1}\al_2\be_2^{-1},r,\eta\rb$.

For (4), the base case $M(1)\ot M(s)\cong P^{\op\lc s/2\rc}\op P_{--}^{\op\lf s/2\rf}\op M(s+1)_{-+}$ should be as follows,
\[\xymatrix@=1em{
&{\stt{u_1\ot u'_i}}\ar@{->}[ld] \ar@{.>}[rd]&\\
{\stt{v_1\ot u'_i+u_1\ot v'_i}}\ar@{.>}[rd]&&{\stt{v_2\ot u'_i+u_1\ot v'_{i+1}}}\ar@{->}[ld]\\
&{\stt{v_2\ot v'_i+v_1\ot v'_{i+1}}}&}\]
for $i=1,\dots,s$. Meanwhile,
\[\xymatrix@=1em{
&\stt{u_1\ot v'_1}\ar@{->}[ld] \ar@{.>}[rd]& &\stt{v_2\ot u'_1}\ar@{->}[ld] &\cdots&\stt{v_2\ot u'_s}\ar@{->}[ld] \ar@{.>}[rd]&\\
\stt{v_1\ot v'_1}&&\stt{v_2\ot v'_1}&\cdots&\stt{-v_2\ot v'_s}&&\stt{-v_2\ot v'_{s+1}}}.\]
Now we prove the general case $M(r)\ot M(s)$ by induction on $r$. From (2) and the base case, we see that
\[\begin{split}
(M(1)&\ot M(r)) \ot M(s)\cong\lb P^{\op \lc r/2\rc}\op P_{--}^{\op \lf r/2\rf}\op M(r{+}1)_{-+}\rb\ot M(s)\\
&\cong P^{\op \lc rs/2\rc}\op P_{--}^{\op \lf rs/2\rf}\op P_{+-}^{\op \lf r(s+1)/2\rf}\op P_{-+}^{\op \lc r(s+1)/2\rc}\op(M(r{+}1)_{-+}\ot M(s)).
\end{split}\]
On the other hand, using (2) and the induction hypothesis, we have
\[\begin{split}
M(1)&\ot(M(r)\ot M(s))\cong M(1)\ot\lb P^{\op \lc rs/2\rc}\op P_{--}^{\op \lf rs/2\rf}\op M(r{+}s)_{-+}\rb\\
&\cong \lb P^{\op \lc rs/2\rc}\op P_{--}^{\op \lf rs/2\rf}\op P_{+-}^{\op rs}\op P_{-+}^{\op rs}\rb\op\lb P_{+-}^{\op\lf (r+s)/2\rf}\op P_{-+}^{\op\lc (r+s)/2\rc}\op M(r{+}s{+}1)\rb.
\end{split}\]
By the Krull-Schmidt theorem, they combine to give $P^{\op\lc (r+1)s/2\rc}\op P_{--}^{\op\lf (r+1)s/2\rf}\op M(r{+}s{+}1)_{-+}$. One can prove the decompositions (5) and (6) by induction similarly.

For (7), we first prove the base case $M(1)\ot C(s,\eta)$. The decomposition is given by
\[\xymatrix@=0.1em{
&\stt{u_1\ot u'_{2i-1}}\ar@{->}[ld] \ar@{.>}[rd]&\\
{\begin{split}
&\stt{u_1\ot v'_{2i-1}}\\
&\stt{+v_1\ot u'_{2i-1}}
\end{split}}\ar@{.>}[rd]&&{\begin{split}
&\stt{u_1\ot(\eta v'_{2i}+v'_{2i-2})}\\
&\stt{+v_2\ot u'_{2i-1}}
\end{split}}\ar@{->}[ld]\\
&{\stt{v_1\ot(\eta v'_{2i}+v'_{2i-1})+v_2\ot v'_{2i-1}}}&}\,
\xymatrix@=1em{
&\stt{u_1\ot u'_{2i}}\ar@{->}[ld] \ar@{.>}[rd]&\\
{\stt{u_1\ot v'_{2i}+v_1\ot u'_{2i}}}\ar@{.>}[rd]&&{\stt{u_1\ot v'_{2i-1}+v_2\ot u'_{2i}}}\ar@{->}[ld]\\
&{\stt{v_1\ot v'_{2i-1}+v_2\ot v'_{2i}}}&}
\]
for $i=1,\dots,s$, and
\[\xymatrix@C=1em{
\stt{v_1\ot u_1'}\ar@{->}[d] \ar@{.>}[rd]|<<<\eta&\stt{-v_1\ot u'_2}\ar@{->}[d] \ar@{.>}[ld] &\stt{v_1\ot u'_3}\ar@{->}[d]\ar@{.>}[rd]|<<<\eta\ar@{.>}[ld]
&\stt{-v_1\ot u'_4}\ar@{->}[d] \ar@{.>}[ld]&\cdots&\stt{v_1\ot u'_{2s-1}}\ar@{->}[d] \ar@{.>}[rd]|<<<\eta&\stt{-v_1\ot u'_{2s}}\ar@{->}[d] \ar@{.>}[ld]\\
\stt{-v_1\ot v'_1}&\stt{v_1\ot v'_2}&\stt{-v_1\ot v'_3}&\stt{v_1\ot v'_4}&\cdots&\stt{-v_1\ot v'_{2s-1}}&\stt{v_1\ot v'_{2s}}}.\]
Now one can prove the general case by induction on $r$ via (1), (2) and (4).

For (8), it is complicated to describe the decomposition even for the base case $W(1)\ot C(s,\eta)$. Instead, we use (6), (7) to deal with it.
\[\begin{split}
W(1)\ot(M(1)\ot C(s,\eta))&\cong W(1)\ot(P^{\op s}\op P_{--}^{\op s}\op C(s,\eta)_{-+})\\
&\cong P^{\op 2s}\op P_{--}^{\op 2s}\op P_{-+}^{\op s}\op P_{+-}^{\op s}\op (W(1)\ot C(s,\eta)_{-+}).
\end{split}
\]
On the other hand, we have
\[\begin{split}
(W(1)\ot(M(1))\ot C(s,\eta)&\cong (P\op P_{--}\op S_{+-})\ot C(s,\eta)_{-+}\\
&\cong P^{\op 2s}\op P_{--}^{\op 2s}\op P_{-+}^{\op 2s}\op P_{+-}^{\op 2s}\op C(s,\eta)_{+-}.
\end{split}\]
Hence, by the Krull-Schmidt theorem, they combine to give
\[W(1)\ot C(s,\eta)\cong P^{\op s}\op P_{--}^{\op s}\op C(s,\eta)_{--}.\]
Now one can prove the general case by induction on $r$ via (1), (2) and (5).

For (9), the base case $M(1)\ot N(s)$ is as follows.
\[\xymatrix@=1em{
&{\stt{u_1\ot u'_i}}\ar@{->}[ld] \ar@{.>}[rd]&\\
{\stt{u_1\ot v'_i+v_1\ot u'_i}}\ar@{.>}[rd]&&{\stt{u_1\ot v'_{i+1}+v_2\ot u'_i}}\ar@{->}[ld]\\
&{\stt{v_1\ot v'_{i+1}+v_2\ot v'_i}}&}\] for $i=1,\dots,s$, where we
take $v'_{s+1}=0$ for simplicity. Meanwhile,
\[\xymatrix@=1em{
&\stt{v_1\ot u'_1}\ar@{->}[ld]^-{-1} \ar@{.>}[rd]& &\stt{v_1\ot u'_2}\ar@{->}[ld]^-{-1} &\cdots&\stt{v_1\ot u'_s}\ar@{->}[ld]^-{-1}\\
\stt{v_1\ot v'_1}&&\stt{v_1\ot v'_2}&\cdots&\stt{v_1\ot v'_s}&}.\]
Now we decompose $M(r)\ot N(s)$ by induction on $r$. First, from (3)
and (4), we have
\[\begin{split}
(M(1)&\ot M(r))\ot N(s)\cong(P^{\op\lc r/2\rc}\op P_{--}^{\op\lf
r/2\rf}\op M(r{+}1)_{-+})\ot N(s)\\
&\cong\lb P^{\op\lc rs/2\rc}\op P_{--}^{\op\lf rs/2\rf}\op
P_{+-}^{\op\lf rs/2\rf}\op P_{-+}^{\op\lc
rs/2\rc}\rb\op(M(r{+}1)_{-+}\ot N(s)).
\end{split}\] On the other hand, by induction and
(2), we get
\[\begin{split}
M(1)&\ot(M(r)\ot N(s))\cong M(1)\ot\lb P^{\op\lc rs/2\rc}\op
P_{--}^{\op\lf rs/2\rf}\op N(s)_{-+}\rb\\
&\cong\lb P^{\op\lc rs/2\rc}\op P_{--}^{\op\lf rs/2\rf}\op
P_{+-}^{\op rs}\op P_{-+}^{\op rs}\rb\op\lb P_{+-}^{\op\lf
s/2\rf}\op P_{-+}^{\op\lc s/2\rc}\op N(s)\rb.
\end{split}\] They combine to give
$M(r{+}1)\ot N(s)\cong P^{\op\lc (r+1)s/2\rc}\op P_{--}^{\op\lf
(r+1)s/2\rf}\op N(s)_{-+}$. The case of $M(r)\ot N'(s)$ is quite
similar to the base case $M(1)\ot N'(s)$ as follows.
\[\xymatrix@=1em{
&{\stt{u_1\ot u'_i}}\ar@{->}[ld] \ar@{.>}[rd]&\\
{\stt{u_1\ot v'_{i-1}+v_1\ot u'_i}}\ar@{.>}[rd]&&{\stt{u_1\ot v'_i+v_2\ot u'_i}}\ar@{->}[ld]\\
&{\stt{v_1\ot v'_i+v_2\ot v'_{i-1}}}&}\] for $i=1,\dots,s$, where we
take $v'_0=0$ for simplicity. Meanwhile,
\[\xymatrix@=1em{
\stt{v_2\ot u'_1}\ar@{.>}[rd]_-{-1}&&\stt{v_2\ot u'_2}\ar@{->}[ld]\ar@{.>}[rd]_-{-1} &\cdots&\stt{v_2\ot u'_s}\ar@{.>}[rd]_-{-1}&\\
&\stt{v_2\ot v'_1}&&\stt{v_2\ot v'_2}&\cdots&\stt{v_2\ot v'_s}}.\]
The decomposition (10) is also straightforward to check.

For (11), when $\eta\neq\ga$, the decomposition can be given by
\[\begin{array}{l}
\xymatrix@=0.5em{
&\stt{u_{2i-1}\ot u'_{2j-1}}\ar@{->}[ld] \ar@{.>}[rd]&\\
{\begin{split}
&\stt{-u_{2i-1}\ot v'_{2j-1}}\\
&\stt{+v_{2i-1}\ot u'_{2j-1}}
\end{split}}\ar@{.>}[rd]&&{\begin{split}
&\stt{-u_{2i-1}\ot(\ga v'_{2j}+v'_{2j-2})}\\
&\stt{+(\eta v_{2i}+v_{2i-2})\ot u'_{2j-1}}
\end{split}}\ar@{->}[ld]\\
&{\begin{split}
&\stt{-v_{2i-1}\ot(\ga v'_{2j}+v'_{2j-2})}\\
&\stt{-(\eta v_{2i}+v_{2i-2})\ot v'_{2j-1}}
\end{split}}&}\,
\xymatrix@=0.5em{
&\stt{u_{2i-1}\ot u'_{2j}}\ar@{->}[ld] \ar@{.>}[rd]&\\
{\begin{split}
&\stt{-u_{2i-1}\ot v'_{2j}}\\
&\stt{+v_{2i-1}\ot u'_{2j}}
\end{split}}\ar@{.>}[rd]&&
{\begin{split}
&\stt{-u_{2i-1}\ot v'_{2j-1}}\\
&\stt{+(\eta v_{2i}+v_{2i-2})\ot u'_{2j}}
\end{split}}\ar@{->}[ld]\\
&{\begin{split}
&\stt{-v_{2i-1}\ot v'_{2j-1}}\\
&\stt{-(\eta v_{2i}+v_{2i-2})\ot v'_{2j}}
\end{split}}&}\\
\xymatrix@C=-4em{
&\stt{u_{2i}\ot u'_{2j-1}}\ar@{->}[ld] \ar@{.>}[rd]&\\
{\stt{u_{2i}\ot v'_{2j-1}+v_{2i}\ot u'_{2j-1}}}\ar@{.>}[rd]&&{\stt{u_{2i}\ot(\ga v'_{2j}+v'_{2j-2})+v_{2i-1}\ot u'_{2j-1}}}\ar@{->}[ld]\\
&{\stt{v_{2i}\ot(\ga v'_{2j}+v'_{2j-2})+v_{2i-1}\ot v'_{2j-1}}}&}\,
\xymatrix@C=-2em{
&\stt{u_{2i}\ot u'_{2j}}\ar@{->}[ld] \ar@{.>}[rd]&\\
{\stt{u_{2i}\ot v'_{2j}+v_{2i}\ot u'_{2j}}}\ar@{.>}[rd]&&{\stt{u_{2i}\ot v'_{2j-1}+v_{2i-1}\ot u'_{2j}}}\ar@{->}[ld]\\
&{\stt{v_{2i-1}\ot v'_{2j}+v_{2i}\ot v'_{2j-1}}}&}
\end{array}
\]
for $i=1,\dots,r,~j=1,\dots,s$.

When $\eta=\ga$, we assume $r\leq s$, without loss of generality.
First, there exist the following two submodules isomorphic to $C(r,\eta)_{--},C(r,\eta)_{-+}$, respectively.
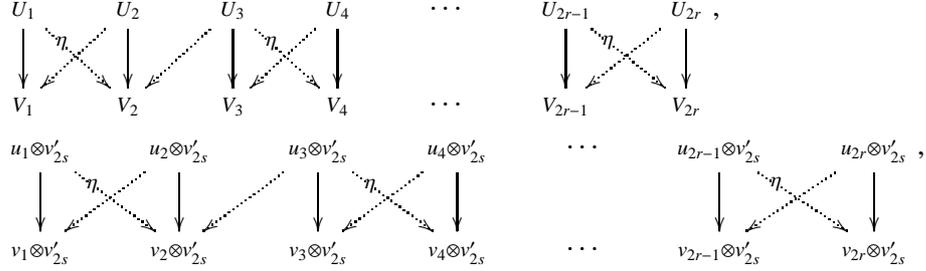
\[\begin{array}{l}
\xymatrix@=2em{
\stt{U_1}\ar@{->}[d] \ar@{.>}[rd]|<<<\eta&\stt{U_2}\ar@{->}[d] \ar@{.>}[ld] &\stt{U_3}\ar@{->}[d]\ar@{.>}[rd]|<<<\eta\ar@{.>}[ld]&\stt{U_4}\ar@{->}[d]
 \ar@{.>}[ld]&\cdots&\stt{U_{2r-1}}\ar@{->}[d] \ar@{.>}[rd]|<<<\eta&\stt{U_{2r}}\ar@{->}[d] \ar@{.>}[ld]\\
\stt{V_1}&\stt{V_2}&\stt{V_3}&\stt{V_4}&\cdots&\stt{V_{2r-1}}&\stt{V_{2r}}},\\
\xymatrix@=2em{
\stt{u_1\ot v'_{2s}}\ar@{->}[d] \ar@{.>}[rd]|<<<\eta&\stt{u_2\ot v'_{2s}}\ar@{->}[d] \ar@{.>}[ld] &\stt{u_3\ot v'_{2s}}\ar@{->}[d]\ar@{.>}[rd]|<<<\eta\ar@{.>}[ld]
&\stt{u_4\ot v'_{2s}}\ar@{->}[d] \ar@{.>}[ld]&\cdots&\stt{u_{2r-1}\ot v'_{2s}}\ar@{->}[d] \ar@{.>}[rd]|<<<\eta&\stt{u_{2r}\ot v'_{2s}}\ar@{->}[d] \ar@{.>}[ld]\\
\stt{v_1\ot v'_{2s}}&\stt{v_2\ot v'_{2s}}&\stt{v_3\ot v'_{2s}}&\stt{v_4\ot v'_{2s}}&\cdots&\stt{v_{2r-1}\ot v'_{2s}}&\stt{v_{2r}\ot v'_{2s}}},
\end{array}\]
where for $k=1,\dots,r$,
\[\begin{split}
U_{2k-1}&=\sum\limits_{i=1}^ku_{2i-1}\ot u'_{2(k-i)+1}+\eta u_{2i}\ot u'_{2(k-i)+2}+\sum\limits_{i=1}^{k-1}u_{2i}\ot u'_{2(k-i)},\\
U_{2k}&=\sum\limits_{i=1}^ku_{2i-1}\ot u'_{2(k-i)+2}+u_{2i}\ot u'_{2(k-i)+1},\\
V_{2k-1}&=\sum_{i=1}^kv_{2i-1}\ot u'_{2(k-i)+1}-u_{2i-1}\ot v'_{2(k-i)+1}
+\eta\,\Big(v_{2i}\ot u'_{2(k-i)+2}+u_{2i}\ot v'_{2(k-i)+2}\Big)\\
&\quad+\sum_{i=1}^{k-1}v_{2i}\ot u'_{2(k-i)}+u_{2i}\ot v'_{2(k-i)},\\
V_{2k}&=\sum\limits_{i=1}^kv_{2i-1}\ot u'_{2(k-i)+2}-u_{2i-1}\ot v'_{2(k-i)+2}+v_{2i}\ot u'_{2(k-i)+1}+u_{2i}\ot v'_{2(k-i)+1}.
\end{split}\]
One can check that $y\cdot U_{2k-1}=\eta V_{2k}+V_{2k-2}$. Meanwhile,
we choose $4rs-2r$ projective submodules listed above for $i=1,\dots,r,~j=2,\dots,s$.
One should examine that all vectors in the diagrams of the chosen submodules form a basis of $C(r,\eta)\ot C(s,\eta)$.
Hence, They combine to give the decomposition
\[C(r,\eta)\ot C(s,\eta)\cong P^{\op(2rs-r)}\op P_{--}^{\op(2rs-r)}\op C(r,\eta)_{--}\op C(r,\eta)_{-+}.\]

For (13), we assume that $r\leq s$ without loss of generality. It gives
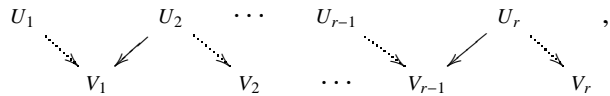
\[
\xymatrix@=1em{
\stt{U_1}\ar@{.>}[rd]&&\stt{U_2}
\ar@{->}[ld] \ar@{.>}[rd] &\cdots&\stt{U_{r-1}}\ar@{.>}[rd]&&\stt{U_r}\ar@{->}[ld] \ar@{.>}[rd]&\\
&\stt{V_1}&&\stt{V_2}&\cdots&\stt{V_{r-1}}&&\stt{V_r}},
\]
where $U_k=\sum_{i=1}^ku_i\ot u'_{k+1-i},V_k=\sum_{i=1}^k v_i\ot u'_{k+1-i}+(-1)^{i-1}u_i\ot v'_{k+1-i},~k=1,\dots,r$. Meanwhile,
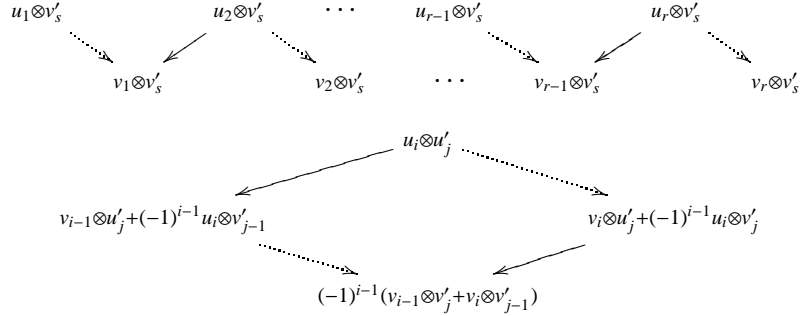
\[
\xymatrix@=1em{
\stt{u_1\ot v'_s}\ar@{.>}[rd]&&\stt{u_2\ot v'_s}\ar@{->}[ld] \ar@{.>}[rd] &\cdots&\stt{u_{r-1}\ot v'_s}\ar@{.>}[rd]&&\stt{u_r\ot v'_s}\ar@{->}[ld] \ar@{.>}[rd]&\\
&\stt{v_1\ot v'_s}&&\stt{v_2\ot v'_s}&\cdots&\stt{v_{r-1}\ot v'_s}&&\stt{v_r\ot v'_s}}.\]
\[\xymatrix@=1em{
&{\stt{u_i\ot u'_j}}\ar@{->}[ld] \ar@{.>}[rd]&\\
{\stt{v_{i-1}\ot u'_j+(-1)^{i-1}u_i\ot v'_{j-1}}}\ar@{.>}[rd]&&{\stt{v_i\ot u'_j+(-1)^{i-1}u_i\ot v'_j}}\ar@{->}[ld]\\
&{\stt{(-1)^{i-1}(v_{i-1}\ot v'_j+v_i\ot v'_{j-1})}}&}\]
for $i=1,\dots,r,j=2,\dots,s$, where we take $v_0=0$ for simplicity.

The decompositions (12) and (14) are similar to prove.
\end{proof}

\begin{cor}\label{gre1}
The Green ring $r(\mh)$ of $\mh$ is a commutative ring generated by \[\Bigl\{\, [S(-,-)],\, [S(+,-)],\, [P(+,+)],\, [M(1)],
\, [W(1)]\,\Bigr\}\cup\Bigl\{\,[C(r,\eta)],\, [N(r)],\, [N'(r)]\,\Bigr\}_{r\in \mathbb{Z}^+,\eta\in\bK^\times},\]
subject to the following relations
\[\begin{array}{l}
S^2=S_-^2=1,\quad C_{r,\eta}S=C_{r,\eta},\\
P^2=(1{+}S)(1{+}S_-)P,\quad MP=(1{+}S_-{+}SS_-)P,\\
WP=(1{+}S{+}S_-)P,\quad C_{r,\eta}P=r(1{+}S{+}S_-{+}SS_-)P,\\
N_rP=\lc r/2\rc(1{+}SS_-)P{+}\lf r/2\rf(S{+}S_-)P,\\
N'_rP=\lc r/2\rc(1{+}S_-)P{+}\lf r/2\rf(S{+}SS_-)P,\\
MW=(1{+}S)P{+}S_-,\\
MN_r=\lc r/2\rc P+\lf r/2\rf SP+SS_-N_r,\\
MN'_r=\lc r/2\rc P+\lf r/2\rf SP+S_-N'_r,\\
WN_r=\lc r/2\rc P+\lf r/2\rf SP+SN_r,\\
WN'_r=\lf r/2\rf P+\lc r/2\rc SP+N'_r,\\
MC_{r,\eta}=r(1{+}S)P{+}S_-C_{r,\eta},\quad WC_{r,\eta}=r(1{+}S)P{+}C_{r,\eta},\\
C_{r,\eta}C_{s,\ga}=\begin{cases}
2rs(1{+}S)P,&\eta\neq\ga,\\
(2rs-\ull{r,s})(1{+}S)P+(1{+}S_-)C_{\ull{r,s},\eta},&\eta=\ga.
\end{cases}\\
N_rN_s=\lc(rs-\ull{r,s})/2\rc P+\lf(rs-\ull{r,s})/2\rf SP+(S^{\ol{r,s}-1}+SS_-)N_{\ull{r,s}},\\
N'_rN'_s=\lf(rs-\ull{r,s})/2\rf P+\lc(rs-\ull{r,s})/2\rc SP+(1{+}S^{\ol{r,s}-1}S_-)N'_{\ull{r,s}},\\
N_rN'_s=\lc rs/2\rc P+\lf rs/2\rf SP,\\
C_{r,\eta}N_s=C_{r,\eta}N'_s=N_rN'_s=rs(1{+}S)P,
\end{array}\]
where we abbreviate $[S(-,-)],[S(+,-)],[P(+,+)],[M(1)],[W(1)],[C(r,\eta)],[N(r)],[N'(r)]$ to $S,\, S_-,\, P,\, M,\, W,\, C_{r,\eta},\, N_r,\, N'_r$, successively.
\end{cor}

The projective class algebra $p(\mh)=\bK[S,S_-,P]/(S^2{-}1,S_-^2{-}1,P^2{-}(1{+}S)(1{+}S_-)P)$.
The Jacobson radical $J(p(\mh))=\lb(1{-}S)P,\,(1{-}S_-)P\rb$,
thus
\[p(\mh)/J(p(\mh))=\bK[S,S_-,P]/(\,S^2{-}1,\,S_-^2{-}1,\,P^2{-}4P,\,(1{-}S)P,\,(1{-}S_-)P\,)
\cong\bK^5,\]
with the orthogonal idempotents $(1{-}S)(1{\pm} S_-)/4,\,(1{+}S)(1{-}S_-)/4,\,((1{+}S)(1{+}S_-){-}P)/4$, $P/4$.
Meanwhile, we can also calculate the Jacobson radical of $R(\mh)$.
\begin{theorem}
$J(R(\mh))$ is equal to
\[\lb(1{-}S)P,\,(1{-}S_-)P,\,(S^{r-1}{-}SS_-)N_r,\,(1{-}S^{r-1}S_-)N'_r,
\,(1{-}S_-)C_{r,\eta}\;\Big|\;r\in\bZ^+,\,\eta\in\bK^\times\rb.\]
\end{theorem}
\begin{proof}
Again we first deal with $\mb{St}(\mh)$. Take \[J=\lb(S^{r-1}{-}SS_-)N_r,\,(1{-}S^{r-1}S_-)N'_r,
\,(1{-}S_-)C_{r,\eta}\;\Big|\;r\in\bZ^+,\,\eta\in\bK^\times\rb\]
and $A=\mb{St}(\mh)/J$. Clearly, $J$ is nilpotent and it reduces to show that $A$ is Jacobson semisimple. 
It is easy to see that $(C_{r+1,\eta}-C_{r,\eta})/2,\,r\in\bN$ are orthogonal idempotents in $A$. For the complicated cases of $N_r, N'_r$, 
we consider the following alternating generators. Let
\[\til_r:=N_r-SN_{r-1},\ \til'_r:=N'_r-SN'_{r-1},\quad r\in\bZ^+,\]
where $N_0=N'_0=0$. One can check that for any $r, t\in\bZ^+$,
\[
\til_r\til_t=[(1+\de_{r,t})(1+S)-2]S_-\til_{\ull{r,t}},\,
\til'_r\til'_t=[2-(1-\de_{r,t})(1+S)]\til'_{\ull{r,t}},
\]
and thus
\[\begin{array}{l}
(\til_r-\til_{r-1})(\til_t-\til_{t-1})=\begin{cases}
2S_-(S\til_r+\til_{r-1}),&r=t,\\
-(1+S)S_-\til_{\ull{r,t}},&|r-t|=1,\\
0,&|r-t|>1.
\end{cases}\\
(\til'_r-\til'_{r-1})(\til'_t-\til'_{t-1})=\begin{cases}
2(\til'_r+S\til'_{r-1}),&r=t,\\
-(1+S)\til'_{\ull{r,t}},&|r-t|=1,\\
0,&|r-t|>1.
\end{cases}
\end{array}\]
Now it is easy to deduce that
\[S_-(1{+}S)\til_r/4,\,S_-(S{-}1)(\til_r{-}\til_{r-1})/4,\,(1{+}S)\til'_r/4,\,(1{-}S)(\til'_r{-}\til'_{r-1})/4,\,r\in\bZ^+\]
are orthogonal idempotents in $A$.
Let $I$ be the ideal of $A$ generated by
\[\big\{\,N_r,\,N'_r,\,C_{r,\eta}\,\mid\,r\in\bN,\eta\in\bK^\times\,\big\}\]
and $\bar{A}:=A/I$, then $\bar{A}$ is generated by $S, S_-, M, W$ and isomorphic to $\bK[\bZ\times\bZ_2\times\bZ_2]$. $\bar{A}$ is clearly Jacobson semisimple and thus $J(A)\subseteq I$.
Now we find the following $\bK$-basis of $I$,
\[\big\{\,S_-(1{+}S)\til_r,\,S_-(1{-}S)(\til_r{-}\til_{r-1}),\,
(1{+}S)\til'_r,\,(1{-}S)(\til'_r{-}\til'_{r-1}),\,
C_{r,\eta}{-}C_{r-1,\eta}\,\big\}_{r\in\bZ^+,\eta\in\bK^\times}\]
which consists of orthogonal idempotents (up to scalars). Hence, we can obtain that $J(A)=J(A)\cap I=0$ exactly as in Theorem \ref{rad}. Furthermore, $J(R(\mh))$ is contained in
\[\lb(1{-}S)P,\,(1{-}S_-)P,\,(S^{r-1}{-}SS_-)N_r,\,(1{-}S^{r-1}S_-)N'_r,
\,(1{-}S_-)C_{r,\eta}\;\Big|\;r\in\bZ^+,\,\eta\in\bK^\times\rb.\]
In order to get the desired result of $J(R(\mh))$,
one only needs to check that $(1{+}S)(1{+}S_-)P/16$ is an idempotent and the ideal $((1{+}S)(1{+}S_-)P)=\bK(1{+}S)(1{+}S_-)P$, thus $((1{+}S)(1{+}S_-)P)\cap J(R(\mh))=0$.
\end{proof}

\subsection{Diverse information list on the Green rings}
Finally, in order to get a clear overview on the differences of the Green rings of $\mh$, $\bar{\ma}$ and $D(H_4)$
from the point of ring structural view, we summarize all computation results as follows:

\noindent
\begin{flushleft}
\renewcommand{\arraystretch}{1.2}

\begin{tabular}{|c|c|c|c|c|c|}
\hline Hopf alg. & Green ring & (co)quasitri. & blocks & proj. class alg. &
modulo rad. \\\hline $\mh$ &  comm. & QT, \,CQT & 1 &
$\bK[s,\,s_-,\,p]/\atop\lb
s^2{-}1,\,s_-^2{-}1,\,p^2{-}(1{+}s)(1{+}s_-)p\rb$ &
$\bK^5$\\\hline $\bar{\ma}$ & non-comm. & CQT & 2 &
$\bK[s,\,s_-,\,p]/\atop\lb s^2{-}1,\, s_-^2{-}1,\,p^2-2(1{+}s)p\rb$
&$\bK^6$ \\\hline $D(H_4)$ & comm. & QT, \,CQT &
3 & $\bK[s,\,p_+]/\atop\lb s^2{-}1,\,p_+^3{-}2(1{+}s)p_+\rb$ &
$\bK^4$\\\hline
\end{tabular}

\smallskip
\begin{tabular}{|c|c|}
\hline Hopf alg. &  Jacobson rad. of the Green alg.\\\hline
$\mh$ & $\begin{array}{l}
\big((1{-}s)p,\,(1{-}s_-)p,\,(s^{r-1}{-}ss_-)n_r,
(1{-}s^{r-1}s_-)n'_r,\,(1{-}s_-)c_{r,\eta}\,\big)_{r\in\bZ^+,\,\eta\in\bK^\times}
\end{array}$\\\hline
$\bar{\ma}$ &
$\lb(1{-}s)p,\,(1{-}s)n_r,\,(1{-}s)n'_r,\,(1{-}s)c_{r,\eta}\,\rb_{r\in\bZ^+,\eta\in\bK^\times}$\\\hline
$D(H_4)$ &
$\lb(1{-}s)p_+,\,(1{-}s)n_r,\,(1{-}s)n'_r,\,(1{-}s)c_{r,\eta}\,\rb_{r\in\bZ^+,\eta\in\bK^\times}$\\\hline
\end{tabular}

\end{flushleft}
\noindent
$\bullet$ In the tables, the generators of algebras or ideals stand for the respective modules.

\begin{rem}
Recall that the Sweedler algebra $H_4$ has the Hopf algebra structure defined by
\[\De(a)=a\ot a,\,\De(b)=b\ot1+a\ot b,\,\ve(a)=1,\,\ve(b)=0,\,S(a)=a,\,S(b)=-ba.\]
In \cite[Example 1.5]{Doi}, the author described all the co-quasitriangular structures on $H_4$. The universal $r$-matrices, denoted $\si_\al$, are parameterized by $\al\in\bK$ and defined as follows.
\[\begin{array}{c@{\hspace{2.5pt}}|@{\hspace{2.5pt}}c
@{\hspace{7.5pt}}c@{\hspace{5pt}}c@{\hspace{5pt}}c}
\si_\al & 1 & a & b & ab\\
\hline
1 & 1 & 1 & 0 & 0\\
a & 1 & -1 & 0 & 0\\
b & 0 & 0 & \al & \al\\
ab & 0 & 0 & -\al & \al\\
\end{array}\]
Note that $\si_\al$ satisfies the following condition
\[\si_\al(x_{(1)},y_{(1)})x_{(2)}y_{(2)}
=\si_\al(x_{(2)},y_{(2)})y_{(1)}x_{(1)},\,x,y\in H_4.\]
Hence, one can check that $\si_\al$ is also a Hopf 2-cocycle of $H_4$. By the above condition, we know that the twisted product $x\cdot_{\si_\al}y=yx$, i.e., $H_4^{\si_\al}$ is just the opposite algebra $H_4^{\bs{op}}$. Since  $(H_4^{\bs{op}},\De,S^{-1})\cong(H_4,\De,S)$ as Hopf algebras with $a, b$ mapping to $a, b$ respectively, the $2$-cocycle twist via $\si_\al$ is just trivial. This means that such twist inside $H_4$ doesn't produce any new Hopf algebra structure under taking those twists on $D(H_4)$, $\mh$ and $\bar {\ma}$ as before, which is witnessed by the classification work of Caenepeel-Dascalescu-Raianuet on all pointed Hopf algebras of dimension $16$ in \cite{CDR}.
In particular, such Hopf algebras with the Klein group algebra
as the coradicals among the $5$ iso-classes in their classification list are just those $3$ Hopf algebras $D(H_4)$, $\mh$ and $\bar {\ma}$ considered in this paper.

In particular, the Hopf algebras of dimension $16$ with the Klein group algebra
as the coradicals have $5$ iso-classes: $A_{C_2}\ot\bK[\bZ_2],\,\bar{\ma},\,\bar{\ma}^*,\,D(H_4),\,\mh$. The dual algebra $\bar {\ma}^*$ of $\bar {\ma}$ also has generators $g, h, x, y$, subject to the following relations,
\begin{equation}\label{rel}
\begin{array}{l}
gh=hg,\quad g^2=h^2=1,\\
gx=-xg,\quad gy=-yg,\quad hx=xh,\quad hy=-yh,\\
x^2=y^2=0,\quad xy=-yx.
\end{array}
\end{equation}
The comultiplication $\De$ is defined by
\[\De(g)=g\ot g,\quad \De(h)=h\ot h,\quad \De(x)=x\ot 1+g\ot x,\quad \De(y)=y\ot 1+g\ot y.\]
On the other hand, since $D(H_4)$ is not basic, $D(H_4)^*$ is not pointed, thus not equal to any of the five iso-classes above. Note that any object in $\{\bar{\ma},\,D(H_4),\,\mh\}$ is not isomorphic to that in $\{A_{C_2}\ot\bK[\bZ_2],\,\bar{\ma}^*\}$ as coalgebras, thus fails to be twist-equivalent. Meanwhile, $A_{C_2}\ot\bK[\bZ_2],\,\bar{\ma}^*$ are not twist-equivalent, too.
\end{rem}

\begin{rem}
In the paper \cite{Wa}, the author discussed the Green rings of all eight-dimensional Hopf algebras over $\bK$. Note that if the tensor module categories of two Hopf algebras are monoidally equivalent, then their Green rings are isomorphic, but the converse is not always true since the latter is a decategorification of the former in the sense of Mazorchuk (\cite{Maz}). Some interesting phenomena in this direction have been pointed out in Corollary 1.7 and Remark 1.8 of \cite{Wa} as follows.

For those non-semisimple Hopf algebras of dimension $8$, all their Green rings are
 commutative, and isomorphic if and only if the
 module categories are monoidally equivalent. However, for
 the semisimple cases, the Green rings of $\bK[D_8]$ (the dihedral group algebra), $\bK[Q_8]$ (the quaternion group algebra), $A_8$ (the Kac algebra) are mutually isomorphic, but their module categories are not monoidally equivalent (see \cite{TY}, \cite{Wa, Wa1}, etc.).
\end{rem}

\vskip30pt \centerline{\bf ACKNOWLEDGMENT}
The authors are indebted to the referee for his/her valuable comments on the earlier version of manuscript, in which an error on the commutativity of the Green ring $r(\bar {\ma})$ has been pointed out. We fix it in Lemma \ref{sim}. The comments also inspire us to provide a new detailed proof for the description of the Jacobson radicals of the Green rings and correct some computations. The thanks also go to Guodong Zhou
for the useful conversation on relevant results involving the theory of special
biserial algebras and recommending us a nice book \cite{Erd} on it.

\end{document}